\documentclass[11pt]{article}
\usepackage[letterpaper,margin=1.0in]{geometry}

\usepackage{sty}

\usepackage{authblk}
\author[1]{Ansh Nagda\thanks{Email: \textit{anshnagda@gmail.com}. A large fraction of this work was done when this author was affiliated with the EECS Department at UC Berkeley.}}
\author[2]{Alexander S.\ Wein\thanks{Email: \textit{aswein@ucdavis.edu}. Partially supported by a Sloan Fellowship and NSF CAREER Award CCF-2338091.}}

\affil[1]{Independent}
\affil[2]{Department of Mathematics, UC Davis}

\title{From Weak to Strong Testing in Gaussian Models}

\date{}

\begin{document}

\maketitle

\begin{abstract}
We study the computational complexity of hypothesis testing in the spiked Wigner model, a prototypical model for detecting low-rank structure in a large random matrix. Below the ``BBP'' eigenvalue transition, it is expected that strong detection --- with both type I and II errors vanishing --- requires exponential time. Assuming this as a conjecture, we determine the limits of polynomial-time weak detection, exactly characterizing the possible tradeoffs between type I and II errors. Specifically, the optimal tradeoff is achieved by a particular linear spectral statistic. Thus, the question of weak detection is entirely reduced to that of strong detection.

The proof builds on ideas of Nagda--Raghavendra (2025) and Moitra--Wein (2025). The low-degree likelihood ratio (LDLR) plays a key role: any test that slightly beats the LDLR can be boosted to have an even higher success probability. This leads us to establish a computational analogue of the Neyman--Pearson lemma for a subclass of additive Gaussian models: for a given super-polynomial runtime, the best possible tradeoff between type I and II errors is either the one achieved by thresholding the LDLR, or the trivial tradeoff that results from strong detection.
\end{abstract}

\newpage

\tableofcontents

\newpage

\section{Introduction}

We consider one of the most basic tasks in statistics: simple-versus-simple hypothesis testing. That is, the aim is to decide which of two known distributions $\PP$ (``planted'') and $\QQ$ (``null'') on $\RR^N$ generated a given sample. A \emph{test} is a function $A \randmap \RR^N \to \{0,1\}$, where the ``star'' notation indicates that we allow randomized functions. The classical Neyman--Pearson lemma characterizes the possible tradeoffs between type I error $\QQ(A=1)$ and type II error $\PP(A=0)$. However, the optimal test may be difficult to compute, and our goal is to understand what is achievable by algorithms of limited computation power (say, polynomial runtime) in a high-dimensional regime where $N \to \infty$.

We begin by introducing the \emph{spiked Wigner model} as a motivating example, although some of our results will apply more generally to a class of \emph{additive Gaussian models} introduced later (\Cref{def:agm}).

\subsection{Spiked Wigner and linear spectral statistics}

We introduce the well-studied task of detecting the presence of a rank-1 ``spike'' in a Gaussian Wigner random matrix. Models of this style were first popularized by~\cite{johnstone,bbp,peche,fp}, among others.

\begin{definition}[Spiked Wigner model]
\label{def:wig-intro}
Consider hypothesis testing between the following two distributions $\QQ = \QQ_n$ and $\PP_\lambda = \PP_{\lambda,n}$ over matrices $\RR^{n \times n}$ for $n \in \NN_{\ge 1}$.
\begin{itemize}
    \item Under the ``null'' distribution $\QQ_n$, we observe $Y \sim \GOE(n)$ drawn from the Gaussian Orthogonal Ensemble (GOE), i.e., $Y \in \RR^{n \times n}$ is symmetric with $\{Y_{ij} \,:\, i \le j\}$ drawn independently as $Y_{ii} \sim \cN(0,2/n)$ and $Y_{ij} \sim \cN(0,1/n)$ for $i<j$.
    \item Under the ``planted'' distribution $\PP_{\lambda,n}$, we observe
    \[ Y = \lambda \frac{xx^\top}{\|x\|_2^2} + W \]
    where $\lambda \ge 0$ is the ``signal-to-noise ratio'' (SNR), $W \sim \GOE(n)$ is the ``noise,'' and the ``spike'' $x \in \RR^n$ has entries drawn i.i.d.\ from some ``prior'' $\pi$, a distribution on $\RR$ with mean 0 and variance 1. In the (unlikely) event that $x=0$, define $xx^\top/\|x\|_2^2 := 0$.
\end{itemize}
The model is parametrized by $\lambda$ and $\pi$, which are known to the statistician. We will be interested in the regime $n \to \infty$ with $\lambda$ and $\pi$ held fixed.
\end{definition}

In the notation $\PP_\lambda$ we have suppressed the dependence on $\pi$ because we will always fix $\pi$ but may consider multiple values for $\lambda$. We have included the exact normalizing factor $\|x\|_2^2$, whereas some authors replace this by the deterministic value $n$ that it concentrates around; our choice is for consistency with prior work~\cite{weak-wigner,precise-error} that we rely on, although we do not anticipate any meaningful difference between the two conventions. Our main result will further require $\pi$ to be supported on a finite subset of $\RR$. For concreteness, let us also make this assumption in the discussion that follows, although some results we reference have weaker assumptions (such as compact support or subgaussian). A canonical choice of prior to have in mind is the \emph{sparse Rademacher distribution}: for a parameter $\rho \in (0,1]$,
\begin{equation}\label{eq:sparse-rad}
x_i = \begin{cases} 1/\sqrt{\rho} & \text{with probability } \rho/2, \\ -1/\sqrt{\rho} & \text{with probability } \rho/2, \\ 0 & \text{with probability } 1-\rho. \end{cases}
\end{equation}
Note that we assume $\pi$ does not depend on $n$, so for us, $\rho$ will always be a fixed constant. This differs from much of the literature on \emph{sparse PCA}, where $\rho$ may depend on $n$, e.g.~\cite{BR-reduction,cov-thresh,subexp-sparse}.

One natural criterion for successful testing is \emph{strong detection}, in the following sense.

\begin{definition}\label{def:strong-dist}
Let $\PP = \PP_n$ and $\QQ = \QQ_n$ be distributions on $\RR^N$ for some $N = N_n$. A test $A \randmap \RR^N \to \{0,1\}$ (also depending on $n$) is said to \emph{strongly distinguish} $\PP$ and $\QQ$ if
\[ \QQ(A=0) \to 1 \qquad \text{and} \qquad \PP(A=1) \to 1 \qquad \text{as } n \to \infty. \]
\end{definition}

In the spiked Wigner model, strong detection is easily achieved when $\lambda > 1$ by inspecting the maximum eigenvalue of $Y$~\cite{eval-evec}. This is known as the ``BBP transition,'' named for Baik, Ben Arous, and P{\'e}ch{\'e}~\cite{bbp}. Below the BBP threshold (i.e., $\lambda < 1$), the maximum eigenvalue (and more generally, any function of the spectrum) fails to achieve strong detection~\cite{limitation-spectral}. In general, strong detection is possible when $\lambda > \lambda^*$ for a particular threshold $\lambda^* = \lambda^*(\pi) \le 1$ that depends on the prior $\pi$~\cite{fundamental-limits}. This threshold may be strictly below 1: for instance, this happens for the sparse Rademacher prior~\eqref{eq:sparse-rad} when the parameter $\rho$ is sufficiently small. By the Neyman--Pearson lemma, the statistically optimal test is to threshold the likelihood ratio, but it is not clear how to compute this without enumerating all exponentially-many choices for the spike $x$. It is expected that this computational intractability may be inherent: for $\lambda^* < \lambda < 1$, there is no known polynomial-time test that achieves strong detection. In fact, every known test requires fully exponential runtime $\exp(n^{1-o(1)})$.

Our current toolbox in average-case complexity theory does not allow us to definitively rule out polynomial-time strong detection in the regime $\lambda^* < \lambda < 1$. However, there are some concrete lower bounds against restricted classes of algorithms, showing that certain known approaches cannot work in this regime. Specifically, there are results ruling out approximate message passing~\cite{LKZ-sparse,LKZ-mmse} and low-degree polynomials~\cite{ld-notes}. Later, we will discuss the framework of low-degree polynomials in more detail, as it will play a central role in our work.

To summarize, in the regime $\lambda^* < \lambda < 1$, strong detection is statistically possible but presumed hard for (computationally) efficient algorithms. Our work will focus on this same ``hard'' regime but we will explore what weaker notions of success \emph{can} be achieved by efficient algorithms here. It is not difficult to achieve \emph{some} nontrivial performance: for any $\lambda > 0$, thresholding the trace $\Tr(Y)$ already achieves \emph{weak detection}, i.e., testing with success probability $\ge 1/2 + \eps$ for a constant $\eps > 0$ (depending on $\lambda$ but not $n$). Our aim will be to characterize the exact limits of weak detection, as described by the \emph{receiver operating characteristic (ROC) curve}. This is a function $\phi: [0,1] \to [0,1]$ that describes the tradeoff between type I and II errors for a family of tests:
\begin{itemize}
    \item On the horizontal axis, we plot $\alpha = \QQ(A=1)$, the \emph{type I error rate} or \emph{false positive rate}.
    \item On the vertical axis, we plot $\beta = \PP(A=1)$, the \emph{power} or \emph{true positive rate}. (The \emph{type II error rate} is $1-\beta$.)
\end{itemize}
For a given $\alpha \in [0,1]$, $\phi(\alpha)$ is defined as the largest $\beta$ for which some test $A$ in the given class achieves $(\alpha,\beta)$, meaning $\QQ(A=1) \le \alpha$ and $\PP(A=1) \ge \beta$. We will focus on a question posed by~\cite{precise-error}: \emph{identify the best possible ROC curve for polynomial-time tests (or more generally, tests satisfying some runtime bound)}. Note that strong detection corresponds (asymptotically) to the ``perfect'' ROC curve $\phi(\alpha) = 1$.

Among all known polynomial-time tests, the best ROC curve is achieved by a particular \emph{linear spectral statistic (LSS)}, meaning a quantity of the form $\sum_i f(\mu_i)$ where $\mu_1 \ge \mu_2 \ge \cdots \ge \mu_n$ are the eigenvalues of $Y$ and $f$ is a function $f: \RR \to \RR$. Specifically, the algorithm is to threshold the value
\begin{equation}\label{eq:lss}
\sum_{i=1}^n f_\lambda(\mu_i) \qquad \text{with} \qquad f_\lambda(\mu) = -\log(1-\lambda\mu+\lambda^2)
\end{equation}
where $\lambda$ is the spiked Wigner SNR. By varying the threshold, this achieves (asymptotically as $n \to \infty$) the ROC curve
\begin{equation}\label{eq:phi-lss}
    \phi_\lambda(\alpha) = 1 - \Phi\left[\Phi^{-1}(1-\alpha) - \sqrt{\frac{1}{2}\log\left(\frac{1}{1-\lambda^2}\right)}\right],
\end{equation}
where $\Phi$ is the standard normal CDF. This result can be found in~\cite{weak-wigner}, although the version we need follows from earlier work~\cite[Theorem~1.6 and Remark~1.7]{sk-ferro}, which builds on~\cite{LSS-1,LSS-2}. While the formula~\eqref{eq:phi-lss} may appear opaque, it arises for a simple reason: the statistic $\sum_i f(\mu_i)$ converges to Gaussian under both $\PP_\lambda$ and $\QQ$, where the two Gaussians have the same variance but different means. The specific choice of $f = f_\lambda$ maximally separates the two Gaussians, and throughout when we refer to ``LSS'' we will always assume the optimal function $f_\lambda$ is used. Note that LSS depends only on the spectrum of $Y$, which (by Gaussian rotational invariance) depends only on the magnitude of the rank-1 signal but not its direction. This is reflected in~\eqref{eq:phi-lss}, which depends on $\lambda$ but not $\pi$.

For any $\lambda < 1$, no algorithm is known to beat the ROC curve $\phi_\lambda$ from~\eqref{eq:phi-lss} --- even at a single point --- without taking fully exponential runtime $\exp(n^{1-o(1)})$. Our main result shows that any hypothetical algorithm that beats $\phi_\lambda$ can be boosted to a strong distinguisher that works at a slightly higher SNR.

\begin{theorem}[Main result for spiked Wigner]
\label{thm:main-wig}
Consider the spiked Wigner model (\Cref{def:wig-intro}). Fix constants $0 < \lambda < \lambda' < 1$ and fix a prior $\pi$ that is mean-zero, variance-one, and supported on a finite subset of $\RR$. Fix constants $\alpha \in [0,1]$ and $\epsilon > 0$, and fix an arbitrary growing sequence $d = d_n \to \infty$. Suppose there is a test $A \randmap \RR^{n \times n} \to \{0,1\}$ computable in time $T = T_n$ that beats the ROC curve $\phi_\lambda$ from~\eqref{eq:phi-lss} in the following sense: for all sufficiently large $n$,
\[ \QQ(A=1) \le \alpha \qquad\text{and}\qquad \PP_\lambda(A=1) \ge \phi_\lambda(\alpha) + \eps. \]
Then there is a test computable in time $n^{O(d)} \cdot T$ that strongly distinguishes (\Cref{def:strong-dist}) $\PP_{\lambda'}$ and $\QQ$. The algorithm requires access to the parameters $\lambda,\lambda',\pi,\alpha,\eps,d$.
\end{theorem}

We defer a discussion on the model of computation to \Cref{sec:background-compmodel}. The above reduction blows up the runtime by a slightly super-polynomial factor, $n^{O(d)}$ where $d$ may grow arbitrarily slowly with $n$. This is negligible compared to the best known algorithms for strong detection when $\lambda < 1$, which have runtime $\exp(n^{1-o(1)})$.

We give two interpretations of the result, one conditional and one unconditional:

\begin{itemize}

    \item {\bf Conditional:} Fix a class $T$ of runtimes strictly between polynomial and exponential; for concreteness, say $\exp(O(n^\delta))$ for some specific choice of constant $\delta \in (0,1)$. Suppose we adopt the predominant belief in the field and assume there is no $T$-time algorithm for strong detection when $\lambda < 1$. Assuming this conjecture, our result implies that LSS is optimal in a sharp sense: for any $\lambda < 1$, there is no $T$-time algorithm that asymptotically beats the ROC curve $\phi_\lambda$. This same conclusion was reached by~\cite{precise-error} based on a different conjecture, and we detail the comparison in \Cref{sec:prior-work}.

    \item {\bf Unconditional:} Without assuming any unproven conjectures, our result still narrows down the ROC curve to two possibilities and establishes a tight connection between strong and weak detection. Fix a class of super-polynomial runtimes $T$, such as the one above. Our result implies there is some computational threshold $\lambda^*_\mathrm{comp} \in [0,1]$ such that: (i) for $\lambda > \lambda^*_\mathrm{comp}$, there is a $T$-time algorithm for strong detection,\footnote{Strong detection is monotone in $\lambda$ because you can add extra Gaussian noise to simulate a lower SNR.} and (ii) for $\lambda < \lambda^*_\mathrm{comp}$, the best possible ROC curve for $T$-time algorithms is $\phi_\lambda$. This reduces the question of weak detection to that of strong detection, because once we know the strong-detection threshold $\lambda^*_\mathrm{comp}$, the entire weak-detection landscape is determined. This is significant because we have well-developed tools for probing the computational complexity of strong detection --- namely the low-degree framework discussed later --- but no established analogue for weak detection.
    
\end{itemize}

Notably, the unconditional interpretation applies even for the class of \emph{all} algorithms, with no bound on runtime. In this case, the threshold $\lambda^*_\mathrm{comp}$ coincides with the statistical threshold $\lambda^*$ mentioned earlier. Our conclusion is consistent with the known information-theoretic limits: for $\lambda < \lambda^*$, it is impossible for \emph{any} test to beat the ROC curve of LSS~\cite{fundamental-limits}.

The prior work that most directly inspired our proof ideas is~\cite{NR-clique}, which gives a related self-reducibility result for the planted clique model. We will also use tools from the above-mentioned~\cite{precise-error}, which was the first to study computational hardness for ROC curves. We discuss the comparison to both of these works in \Cref{sec:prior-work}.

What is special about the specific function $\phi_\lambda$ that makes it the only viable ROC curve aside from the perfect one? The answer lies in its connection to the low-degree likelihood ratio. We elaborate in the next section.

\subsection{Computational Neyman--Pearson lemma}
\label{sec:comp-np}

We now discuss in higher generality some of the mechanisms that underlie the above Wigner result. Our second main result will apply to a more general class of additive Gaussian models, of which spiked Wigner is a special case (see \Cref{sec:equiv-gaussian}).

\begin{definition}[Additive Gaussian Model]
\label{def:agm}
Consider hypothesis testing between the following two distributions $\QQ$ and $\PP_\lambda$ over $\RR^N$, where $\lambda \ge 0$ and $N \in \NN_{\ge 1}$.
\begin{itemize}
    \item Under $\QQ$, we observe $Y \sim \cN(0,I_N)$.
    \item Under $\PP_\lambda$, we observe $Y = \lambda X + Z$ where $\lambda \ge 0$ is the ``SNR,'' $Z \sim \cN(0,I_N)$ is the ``noise,'' and the ``signal'' $X$ is some $\RR^N$-valued random variable that is independent from $Z$. We assume $X$ has finite moments of all orders: any polynomial $f: \RR^N \to \RR$ has $\EE|f(X)| < \infty$.
\end{itemize}
The model is parameterized by the choice of $\lambda$ and $X$. We will often consider an asymptotic regime where $n \to \infty$, $N = N_n \to \infty$, $X = X_n$ depends on $n$, but $\lambda$ is held fixed. For each $n$ we assume $X$ has finite moments as above, but there is no restriction on how the moments scale with $n$.
\end{definition}

As usual, we assume the statistician has full knowledge of the distributions $\PP_\lambda$ and $\QQ$, including the distribution of $X$ (for our purposes, the order-$\le d$ moments of $X$ will be enough).

The classical Neyman--Pearson lemma~\cite{np-lemma} states that optimal hypothesis testing is achieved by thresholding the likelihood ratio (LR) $L = d\PP_\lambda/d\QQ$; by varying the threshold, this attains the best possible ROC curve among all tests. A central object for us will be the \emph{low-degree likelihood ratio (LDLR)}, which acts as a computationally-bounded proxy for the LR. We will formally introduce the LDLR in \Cref{sec:background-ldlr}, but briefly, it is the best low-degree polynomial approximation to the LR, or equivalently, the orthogonal projection of the LR onto the vector space of low-degree polynomial functions. The term ``low-degree likelihood ratio'' was coined by~\cite{hopkins-thesis}, building on~\cite{pcal,HS-bayesian,sos-detect}. It has played a key role in average-case complexity theory by capturing, in a sense, the efficacy of tests that compute low-degree polynomial functions $\RR^N \to \RR$. We refer to the surveys~\cite{ld-notes,ld-survey} for additional context.

For a parameter $d$, possibly depending on $n$, the LDLR for testing $\PP$ versus $\QQ$ is the maximizer of the ratio
\[ R^{\PP}(f) := \frac{\EE_{\PP}[f]}{\sqrt{\EE_\QQ[f^2]}} \]
over all (multivariate) polynomial functions $f: \RR^N \to \RR$ of degree $\le d$. Taking $d \to \infty$ (for fixed $n$) recovers the usual LR. We give more background on the LDLR in \Cref{sec:background-ldlr}. Traditionally, the LDLR has been used to make predictions about the computational complexity of strong or weak detection in a coarse sense: if the maximum value of $R^\PP$ is $O(1)$ or $1+o(1)$, this implies that degree-$d$ polynomials cannot achieve\footnote{The formal criterion for ``success'' here is ``separation'' as defined in~\cite[Definition~1.8]{franz-parisi-criterion}.} strong or weak detection (respectively), which heuristically suggests that the detection task requires runtime (roughly) exponential\footnote{For intuition, the runtime to evaluate a degree-$d$ polynomial term-by-term is $N^{O(d)}$, the number of terms.} in $d$. This viewpoint is by now standard and we refer to the survey~\cite{ld-survey} for further exposition.

It is less clear how the LDLR relates to precise error rates in weak detection, particularly the ROC curve. Unlike the LR, the LDLR does not seem to have any \emph{a priori} guarantees about its ROC curve. It is perhaps natural to suspect that thresholding the LDLR should (approximately) produce the best possible ROC curve among all degree-$d$ polynomial threshold tests, but this is not known. In fact, it has proven difficult to analyze the limits of polynomial threshold tests, with a few recent successes~\cite{ptf-ngca,ptf-regular} in regimes not applicable to our additive Gaussian setting.

Our next main result provides some analogue of the Neyman--Pearson lemma for the LDLR: given any algorithm that beats the ROC curve obtained by thresholding the LDLR, we can produce an algorithm for strong detection at a slightly higher SNR. This implies a dichotomy similar to our first result: for a given class of super-polynomial runtimes, the best possible ROC curve is either the one attained by the LDLR, or the perfect one. Again, the reduction incurs a slightly super-polynomial factor in the runtime. This result applies for additive Gaussian models that satisfy two additional properties, which for now we describe informally with the details deferred to \Cref{thm:comp-np}:

\begin{itemize}

\item {\bf Ratio-of-norms bound:} any degree-$d$ polynomial $p: \RR^N \to \RR$ satisfies $\EE_{\PP_\lambda}[p^2] \le \tau \EE_\QQ[p^2]$ for some parameter $\tau = \tau_d$. We will need $d$ to be slowly growing and $\tau$ to scale as $\exp(o(d))$.

\item {\bf Large-set expansion:} in very rough terms, any subset of $\RR^N$ with large enough measure under $\PP_\lambda$ expands to measure nearly 1 when the Gaussian noise operator is applied. Perhaps more instructive than the statement itself is the way in which we will use this assumption: see Step~3 in \Cref{sec:pf-overview}.

\end{itemize}

These two assumptions are satisfied for the spiked Wigner model, as we will show (see \Cref{sec:ratio-of-norms,sec:lse}). Verifying them is not immediate though, and comprises much of the technical work in this paper (\Cref{app:ratio-norms,app:lse}).

\begin{theorem}[Informal, see \Cref{thm:comp-np}]
\label{thm:comp-np-informal}
Consider an additive Gaussian model (\Cref{def:agm}) satisfying the two additional properties mentioned above. Fix constants $0 < \lambda < \lambda'$ and fix an arbitrary sequence $d = d_n \to \infty$. If there is a test computable in time $T$ that beats the ROC curve of the LDLR for testing $\PP_\lambda$ versus $\QQ$, then there is a test computable in time $N^{O(d)} \cdot T$ that strongly distinguishes $\PP_{\lambda'}$ and $\QQ$.
\end{theorem}

While the low-degree framework has already proved to be widely useful, recent work~\cite{ld-conj-false,ld-fails-subspace,poly-ld-false} has also discovered various counterexamples, meaning choices of $\PP$ and $\QQ$ for which the LDLR does \emph{not} predict the right computational complexity because low-degree polynomials are ``beaten'' by some other efficient algorithm. In a way, our work complements this by giving a rigorous result \emph{in support} of the LDLR. For a well-defined class of testing problems, we show that the LDLR gives rise to the only viable \emph{nontrivial} ROC curve.

We take a moment to clarify how \Cref{thm:comp-np} connects to our first result on the spiked Wigner model (\Cref{thm:main-wig}). One potential avenue to prove our Wigner result would be to show that thresholding the LDLR produces the ROC curve $\phi_\lambda$ and then apply \Cref{thm:comp-np}. We expect this to work in principle, but it appears nontrivial to directly determine the ROC curve of the LDLR in the spiked Wigner model, and it turns out this step can be avoided using an existing approach from~\cite{precise-error}. This approach only requires computing the value of the ratio $R^\PP$ of the LDLR rather than its ROC curve, and it also relies on the known positive result for LSS. We will actually give separate proofs of our two main results (\Cref{thm:main-wig,thm:comp-np}), but these proofs share a few common ingredients that are described in the next section. We view \Cref{thm:comp-np} as being mainly of philosophical importance to the LDLR, whereas if one aims to determine the ROC curve for a specific model, it might be easier to instead follow our proof of \Cref{thm:main-wig}.

\subsection{Proof overview}
\label{sec:pf-overview}

The proofs of our two main results share some modular components that can be broken down into three steps. We give an overview below, with the full details of these components in \Cref{sec:general}.

\paragraph{Step 1: from ROC to ratio.}

For both of our results, we begin with a hypothetical test that achieves some point above the ROC curve in question (either that of LSS or LDLR). The goal is to ``boost'' the success probability and achieve strong detection. The first step is to produce a (possibly randomized) ``score'' function $F \randmap \RR^N \to \RR$ that achieves a larger value than the LDLR for the ratio
\begin{equation}\label{eq:ratio}
R^{\PP}(F) := \frac{\EE_{\PP}[F]}{\sqrt{\EE_\QQ[F^2]}} \end{equation}
mentioned above. Equivalently, $F$ achieves a larger ratio than any degree-$d$ polynomial, where $d = d_n$ is some slowly-growing parameter.

The construction of $F$ is the main step at which our two proofs deviate. For the spiked Wigner result, this step already follows from prior work~\cite{precise-error}, which we elaborate on in \Cref{sec:prior-work}. For the computational Neyman--Pearson lemma, we give a direct proof for this step in \Cref{sec:roc-to-ratio}, showing that if you can beat the ROC curve of the LDLR then you can beat the ratio of the LDLR.

\paragraph{Step 2: from ratio to one-sided test.}

Now we have a function $f$ (assume for simplicity it is deterministic) that slightly beats the ratio of the LDLR. The next step is to boost this and achieve a much larger ratio of value $\exp(\Omega(d))$. This step is the heart of the proof and it uses a core idea from~\cite{NR-clique}, which we now describe. Since $f$ has a better ratio than any low-degree polynomial, its ``high-degree'' component $f_{>d}$ must itself achieve a nontrivial ratio: $R^\PP(f_{>d}) = \Omega(1)$. Now apply the Gaussian noise operator to produce a ``noisy'' version of the function $f_{>d}$, called $\mu$, and consider the ratio $R^\PP(\mu)$ as compared to $R^\PP(f_{>d})$. In the numerator, the effect of noise is $\EE_{\PP_{\lambda'}}[\mu] = \EE_{\PP_\lambda}[f_{>d}]$ where $\lambda < \lambda'$, so this change can be absorbed by incurring a small loss in SNR. Crucially, the denominator undergoes a dramatic change because noise attenuates high-degree functions, and this boosts the ratio by a factor $\exp(\Omega(d))$.

Now we have achieved a very large value for the ratio. However, a large ratio does not in general imply a good test because the numerator might only be large due to a rare event. By analyzing the first two moments of $\mu$, we are able to show that thresholding $\mu$ produces a test with a nontrivial one-sided guarantee:
\[ \QQ(A=1) = \exp(-\Omega(d)) \qquad\text{and}\qquad \PP(A=1) = \exp(-o(d)). \]
In words, the guarantee under $\QQ$ is strong: false positives are very rare. Meanwhile, the guarantee under $\PP$ is weak because true positives are also rare, but crucially they are less rare than false positives. The moment bounds that go into this claim require a few ingredients, namely the ratio-of-norms property (see \Cref{sec:comp-np}) along with some regularity properties for $f$ (which are deduced from the way $f$ was originally constructed).

\paragraph{Step 3: from one-sided to two-sided test.}

The final step is to turn our one-sided test into a two-sided test, where both type I and II error probabilities are $o(1)$. To boost the true positive rate, we will need to run the one-sided test on many different inputs in the hope that one of these runs will output 1. How can we produce many different inputs, given that we are only provided with one? A first attempt is to use \emph{Gaussian cloning} (e.g.~\cite[Section~10]{BBH-reduction} or~\cite[Appendix~A.2]{ld-notes}) to produce $m$ independent copies of the input, but each copy would suffer a factor of $\sqrt{m}$ in the SNR. Since we are only willing to lose a $1+\eps$ factor in the SNR, we instead must produce $m$ highly-correlated copies. The key technical challenge is to prove that, even though the copies are not independent, there is a very high chance that the one-sided test will output 1 on at least one copy. We use the large-set expansion property (see \Cref{sec:comp-np}) here, and in fact much of the heavy lifting is done by that property, which was designed exactly for this purpose.

\paragraph{Verifying the assumptions.}

There is one last substantial component of our proofs that we have not yet discussed: how to verify the two assumptions --- ratio-of-norms bound and large-set expansion --- for the spiked Wigner model. We defer this discussion to \Cref{sec:ratio-of-norms,sec:lse}, respectively. Our assumption that $\pi$ has finite support is used in the proof of large-set expansion.

\subsection{Relation to prior work}
\label{sec:prior-work}

The prior work that we build on most closely is~\cite{precise-error} and~\cite{NR-clique}. We give a detailed comparison here.

\paragraph{Comparison to Moitra--Wein~\cite{precise-error}.}

The topic of our first result --- the computational ROC curve for spiked Wigner --- was first studied by~\cite{precise-error}, arriving at the same conclusion as us (optimality of LSS) but conditional on a ``strong'' version of the low-degree conjecture. Our work can be viewed as a strengthening of~\cite{precise-error} that upgrades the conjecture to a more ``standard'' one. The core contribution of~\cite{precise-error} is a connection between the ROC curve and the ratio~\eqref{eq:ratio}, which we also use in our proof of \Cref{thm:main-wig}.

The main technical component of~\cite{precise-error} is the following: suppose we have a family of $T$-time tests achieving a particular ROC curve $\phi$ (subject to some regularity assumptions) along with a $T$-time test achieving one point $(\alpha,\beta)$ above the curve $\phi$. Then this can be transformed into an $O(T)$-time score function $F \randmap \RR^N \to \RR$ whose ratio exceeds the ``value'' of the curve:
\[ R^\PP(F) > \mathrm{val}(\phi) := \sqrt{\int_0^1 (\phi'(t))^2 dt}. \]
In the spiked Wigner model with $\lambda < 1$, the LSS ROC curve satisfies $\mathrm{val}(\phi_\lambda) = (1-\lambda^2)^{-1/4}$, which exactly coincides with the asymptotic value of the LDLR's ratio. As a result,~\cite{precise-error} provides Step~1 of our proof overview in \Cref{sec:pf-overview}.

The main conclusion of~\cite{precise-error} is optimality of LSS, conditional on a new ``strong'' variant of the low-degree conjecture. Specializing to the spiked Wigner setting with $\lambda < 1$, the standard low-degree conjecture asserts that strong detection is hard because the LDLR ratio is $O(1)$, while the strong low-degree conjecture asserts that it is hard to compute a score function $F$ whose ratio beats the precise value of the LDLR ratio, $(1-\lambda^2)^{-1/4}$. However, the strong conjecture turns out to be false: it was refuted by Ansh Nagda in 2025. The issue is that a score function can exploit rare events to make the ratio extremely large. A ``fix'' is proposed in~\cite[arXiv~v3]{precise-error}, where the conjecture is refined so that the score function must output values in $[1/B,B]$ for a constant $B \ge 1$. The original conclusion of~\cite{precise-error} still holds under this refined conjecture, although this series of events sheds doubt on whether any version of the strong conjecture is true at all.

Our work puts things on more solid ground, relying only on the more established conjecture about strong detection, or alternatively, establishing a rigorous connection between the strong and weak detection questions. We also vindicate the refined version of the strong low-degree conjecture mentioned above, showing that it follows from the established conjecture on strong detection. In other words, the standard low-degree conjecture implies the (refined) strong low-degree conjecture. See~\Cref{prop:main-reduction-wigner}, which essentially states that if the refined strong conjecture were to fail, this could be used to build a strong distinguisher for some $\lambda < 1$. Our proof crucially uses the boundedness $F(y) \in [1/B,B]$ in \Cref{sec:ratio-to-one-sided}, so we do not similarly vindicate the original falsified strong conjecture.

\paragraph{Comparison to Nagda--Raghavendra~\cite{NR-clique}.}

The work of~\cite{NR-clique} uses self-reducibility to amplify hardness in the planted clique model, following a line of inquiry pursued by~\cite{self-amp-1,self-amp-2}. The blueprint for Steps~2 and~3 in our proof overview (\Cref{sec:pf-overview}) is based on~\cite{NR-clique}, especially the key idea of boosting a function's ratio by applying noise to its high-degree part. While our result is also a self-reduction, the context differs from~\cite{NR-clique}: they measure success in terms of the \emph{advantage} over a random guess (a single number) in a regime where the advantage is vanishing, whereas we measure success in terms of the ROC curve in a regime where this curve converges to some nontrivial tradeoff. We will now clarify the main differences between our proof and that of~\cite{NR-clique}.

One of the more technical parts of~\cite{NR-clique} is the anticoncentration result in Section~6, which is used when turning a large ratio $R^\PP$ into a one-sided distinguisher. In our proof (Step~2), we have replaced this part of the argument by a simpler and more direct application of Paley--Zygmund. We expect this same idea can be used to substantially simplify the proof of~\cite{NR-clique}.

Furthermore, a few aspects of our setting create technical challenges that were not present in the planted clique setting of~\cite{NR-clique}. First, the ratio-of-norms bound --- $\EE_{\PP}[p^2] \le \tau \EE_\QQ[p^2]$ for all degree-$d$ polynomials $p$ --- from \Cref{sec:comp-np} is an important ingredient both in our proof and that of~\cite{NR-clique}. In the planted clique setting, the ratio $R^\PP$ achieved by the LDLR (or equivalently, the norm of the LDLR; see~\Cref{sec:background-ldlr}) is $1+o(1)$. By a simple and generic argument using the hypercontractive inequality, this implies the ratio-of-norms bound with $\tau = 1+o(1)$ and $d$ slowly growing; see~\cite[Equations (11), (12)]{NR-clique}. In contrast, we are in a scenario where the LDLR ratio converges to some constant strictly greater than 1. We are not aware of any generic argument to bound the ratio of norms in such a case, and we instead give a proof that uses the specifics of the spiked Wigner model. This is one of the more involved arguments in our paper; see \Cref{sec:ratio-of-norms} and \Cref{app:ratio-norms}.

An additional difficulty in our setting is that we care about exact constants in the SNR, so we are only allowed to lose a $1+\eps$ factor in $\lambda$ when boosting a one-sided distinguisher to a two-sided one. In contrast, the planted clique setting of~\cite{NR-clique} has the leeway to lose an $n^\eps$ factor in the clique size. For us, this means we cannot use the trivial boosting strategy --- producing many \emph{independent} copies of the input --- and instead must analyze the one-sided test when run on many \emph{highly-correlated} inputs. This creates the need for the large-set expansion property, and verifying this property for the spiked Wigner model is another of the more involved arguments in our paper; see \Cref{sec:lse} and \Cref{app:lse}.

\subsection{Discussion and future directions}

The arguments we present in this work are fairly generic and may be useful in other settings to determine the computational ROC curve. We note, however, that in some models --- like tensor PCA (e.g.~\cite{kikuchi,kikuchi-smooth}) --- it is possible to increase statistical power by increasing runtime in a smooth fashion. This differs from our spiked Wigner setting, where all runtimes between polynomial and exponential have effectively the same statistical power. As a result, there does not seem to be a meaningful computational question to ask about weak detection in models like tensor PCA, because weak detection can always be boosted to strong detection at the expense of a slight increase in runtime.

One feature of the spiked Wigner problem that was crucial for us is the null distribution with independent coordinates. It is less clear whether our approach could be extended to more complex testing problems, such as planted-versus-planted (e.g.~\cite{count-find,almost-orthonormal,sharp-pvp}) or random regular graphs~\cite{spectral-planting,detecting-lifts,ptf-regular}. A related question in the spiked Wigner model is that of estimating the spike when $\lambda > 1$. The estimation question does not seem to have quite the same dichotomy that we have discovered for testing, since the statistical mean squared error (MSE) may lie strictly between the computational MSE and the ``perfect'' MSE (e.g.~\cite{LKZ-mmse}).

There is an extensive body of work on \emph{average-case reductions} that formally relate different statistical models to each other. See, for instance,~\cite{BBH-reduction,secret-leakage,wigner-wishart} and references within. One might hope to connect our spiked Wigner result to the existing web of reductions, showing that LSS is optimal in the spiked Wigner model conditional on (say) the planted clique hypothesis. The current difficulty with this is that existing reduction-based hardness for the spiked Wigner model does not reach all the way to the sharp threshold $\lambda = 1$.

\subsection{Organization}

\Cref{sec:background} establishes some notation, conventions, and background that will be used throughout. \Cref{sec:general} gives the full details for the three modular components of our reductions from the overview in \Cref{sec:pf-overview}. \Cref{sec:comp-np-formal} gives the formal statement and proof of our computational Neyman--Pearson lemma. \Cref{sec:wigner-application} proves \Cref{thm:main-wig}, the result for the spiked Wigner model, with some technical ingredients deferred to \Cref{app:ratio-norms,app:lse}.

\section{Notation and Background}
\label{sec:background}

\subsection{Asymptotic notation}

We will often consider a sequence of problems indexed by $n$, with some parameters depending on $n$ (e.g.\ $N = N_n$, $X = X_n$) and others designated as fixed ``constants'' which do not depend on $n$ (e.g.\ $\lambda, \pi$). Asymptotic notation like $o(\cdot), O(\cdot), \Omega(\cdot)$ generally pertains to the limit $n \to \infty$, meaning there may be hidden factors depending on constants like $\lambda,\pi$. On the other hand, an ``absolute'' constant means a specific number that does not depend on any parameters.

\subsection{Randomized functions}
\label{sec:background-randf}

We use the notation $F \randmap A\to B$ for a randomized function from $A$ to $B$. This means the output $b \in B$ may depend on both the input $a \in A$ and some internal random seed $\omega \in \Omega$. The seed $\omega$ is drawn from some prescribed distribution over a set $\Omega$, which we write as $\omega \sim \Omega$. We will sometimes write the output as $F_\omega(a)$ to make $\omega$ explicit. Note that every (deterministic) function $f:A\to B$ can be interpreted as a randomized function. In this paper, all functions (randomized or not) are always assumed to be measurable so that we can define probabilities such as $\PP(F \ge t)$; this means the probability over both $Y \sim \PP$ and $\omega \sim \Omega$ of the event $\{F_\omega(Y) \ge t\}$. Similarly, $\EE_\PP[F] := \EE_{Y \sim \PP} \, \EE_{\omega \sim \Omega}\,[F_\omega(Y)]$.

\subsection{Model of computation}
\label{sec:background-compmodel}

We are dealing with computational problems with real-valued inputs, and we will present algorithms in ``exact'' terms, without modeling the round-off errors that would occur on a digital computer. This is a standard abstraction used in the field of average-case reductions (e.g.~\cite{secret-leakage} and references therein). In effect, we are working in a \emph{real RAM} model of computation (see~\cite{BSS}), where an arbitrary real number can be stored in a single memory cell, and basic arithmetic or comparisons can be done in a single operation. However, we will also use standard functions such as $x \mapsto \sqrt{x}$ that are not usually permitted in the real RAM model. For the algorithms described in this paper, we feel it is reasonably clear that there will be no inherent issues with bit complexity or numerical stability, so for instance, the square-root function could be replaced by a polynomial approximation. In principle, we believe our algorithms could be analyzed in finite-precision arithmetic, but this would be very tedious.

We will sometimes deal with randomized algorithms that compute randomized functions. We assume such algorithms have the ability to draw exact samples from the standard Gaussian distribution $\cN(0,1)$. When we speak of an algorithm of runtime $T$, this may stand for a class of runtimes such as $O(n^2)$ or $n^{O(1)}$. Randomized algorithms must \emph{always} (with probability 1) obey this runtime bound. We will generally assume our algorithms have access to parameters such as $\lambda,d$; these can be thought of as additional inputs to the algorithm (or for constants like $\lambda$, they can simply be ``hard-coded'' into the algorithm).

Perhaps the most sophisticated computation used in this paper are eigenvalue computations appearing in the LSS statistic $\sum_i f_\lambda(\mu_i)$ from~\eqref{eq:lss}. However, it is possible to compute the LSS statistic to high accuracy without computing all the eigenvalues $\mu_i$ of $Y$ individually. To do this, compute $\Tr(g(Y))$ where $g$ is a polynomial approximation to $f_\lambda$. The approximation only needs to be accurate on $[-2-\eps,2+\eps]$ for a constant $\eps>0$, as all the eigenvalues lie in this interval with high probability (when $\lambda \le 1$).

\subsection{Gaussian space and Hermite polynomials}
\label{sec:background-gaussian}

We refer to~\cite[Ch~11]{odonnell-book} as a standard reference for the following material. For functions $f,g : \RR^N \to \RR$ and a distribution $\PP$ over $\RR^N$, define the inner product and norm
\[ \langle f,g \rangle_\PP := \EE_\PP[fg] = \EE_{Y \sim \PP}[f(Y)g(Y)], \qquad \|f\|_\PP^2 := \langle f,f \rangle_\PP. \]
Let $L^2(\PP)$ denote the space of functions $f : \RR^N \to \RR$ such that $\|f\|_\PP < \infty$, equipped with the above norm and inner product. For a randomized function $F = F_\omega \randmap \RR^N \to \RR$ with random seed $\omega \sim \Omega$, we may similarly use the notation $\|F\|_\PP^2 := \EE_\PP[F^2] = \EE_{Y \sim \PP} \, \EE_{\omega \sim \Omega} \, [F_\omega(Y)^2]$.

An important special case is the standard Gaussian distribution $\QQ = \cN(0,I_N)$. The \emph{Hermite polynomials}, $\{h_k\}_{k \in \NN}$ where $\NN := \{0,1,2,\ldots\}$, provide an orthogonal basis for the polynomials $\RR \to \RR$ with respect to $\cN(0,1)$. We will use the \emph{orthonormal} variant, that is, $\EE_{z \sim \cN(0,1)}[h_k h_\ell] = \mathbf{1}\{k=\ell\}$. Explicitly, these polynomials can be constructed via the recurrence
\[ h_0(z) = 1, \qquad h_1(z) = z, \qquad h_{k+1}(z) = \frac{z}{\sqrt{k+1}} \, h_k(z) - \sqrt{\frac{k}{k+1}} \, h_{k-1}(z) \; \text{ for } \; k \ge 1. \]
Note that $\deg(h_k) = k$. For multivariate polynomials $\RR^N \to \RR$, the \emph{multivariate Hermite polynomials} provide an orthonormal basis with respect to $\cN(0,I_N)$. These are indexed by $\alpha \in \NN^N$ and constructed by taking products of the univariate Hermite polynomials:
\[ h_\alpha(z) := \prod_{i=1}^N h_{\alpha_i}(z_i). \]
For $\alpha \in \NN^N$ we write $|\alpha| := \sum_{i \in [N]} \alpha_i$ where $[N] := \{1,2,3\ldots,N\}$. Now $\{h_\alpha \,:\, |\alpha| \le d\}$ forms a basis for the degree-$\le d$ polynomials $\RR^N \to \RR$, and this basis is orthonormal with respect to $\langle \cdot,\cdot \rangle_\QQ$, i.e., $\EE_{z \sim \cN(0,I_N)}[h_\alpha h_\beta] = \mathbf{1}\{\alpha=\beta\}$. In particular, since $h_0$ is the constant polynomial $1$, we have $\EE_{z \sim \cN(0,I_N)}[h_\alpha] = 0$ for all $|\alpha| \ge 1$.

For a vector $x \in \RR^N$ and multi-indices $\alpha,\beta \in \NN^N$, we write
\[ x^\alpha := \prod_{i \in [N]} x_i^{\alpha_i}, \qquad \alpha! := \prod_{i \in [N]} \alpha_i!, \qquad \binom{\alpha}{\beta} := \prod_{i \in [N]} \binom{\alpha_i}{\beta_i}. \]
Inequalities between multi-indices are interpreted entrywise: $\beta \le \alpha$ means $\beta_i \le \alpha_i$ for all $i \in [N]$. The following useful formula is standard; see e.g.~\cite[Proposition~3.1]{schramm-wein-estimation}.

\begin{proposition}[Shifted Hermite expansion]
\label{prop:hermite-shift}
For any $\alpha \in \NN^N$ and any $x,z \in \RR^N$,
\[
h_\alpha(x+z)
= \sum_{\beta \le \alpha}
\sqrt{\frac{\beta!}{\alpha!}}
\binom{\alpha}{\beta}
x^{\alpha-\beta} h_\beta(z).
\]
\end{proposition}

For $\rho \in (0,1)$, the \emph{Gaussian noise operator} $T_\rho$ will play a central role for us. For a distribution $\PP$ on $\RR^N$, we write $T_\rho \PP$ for the law of $\rho Y + \sqrt{1-\rho^2}G$ where $Y\sim \PP$ and $G\sim \cN(0,I_N)$ are independent. Note that for an additive Gaussian model (\Cref{def:agm}) with $\rho = \lambda/\lambda' \in (0,1)$, we have $T_\rho \PP_{\lambda'} = \PP_\lambda$. For a function $g: \RR^N \to \RR$, $T_\rho g$ is another function $\RR^N \to \RR$ defined as
\[ T_\rho g(y) := \EE_{G \sim \cN(0,I_N)}\!\left[g(\rho y + \sqrt{1-\rho^2}G)\right]. \]
The noise operator interacts nicely with the Hermite basis. If $p$ is a degree-$\le d$ polynomial with Hermite expansion $p = \sum_{|\alpha|\le d} \hat p_\alpha h_\alpha$ for some coefficients $\hat{p}_\alpha \in \RR$, then
\begin{equation}\label{eq:T-hermite}
T_\rho p = \sum_{|\alpha|\le d} \rho^{|\alpha|} \hat p_\alpha h_\alpha.
\end{equation}

Finally, a standard ``hypercontractive'' inequality on Gaussian space relates the 2-norm and 4-norm of a low-degree polynomial; see~\cite{odonnell-book}, Theorem~9.21 together with the discussion in Section~11.1 (Gaussians can be simulated as a sum of Rademachers).

\begin{proposition}[\cite{bonami,beckner,gross-logsob}]
\label{prop:gauss-bonami}
For a polynomial $f: \RR^N \to \RR$ of degree $\le d$ with input from $\QQ = \cN(0,I_N)$, 
\[ \EE_\QQ[f^4] \le 9^d \left(\EE_\QQ[f^2]\right)^2. \]
\end{proposition}

\subsection{Low-degree likelihood ratio}
\label{sec:background-ldlr}

We now give a formal introduction to the low-degree likelihood ratio (LDLR), which was discussed earlier. We refer to e.g.~\cite{hopkins-thesis,ld-notes,ld-survey} for further exposition.

\begin{definition}[Low-degree likelihood ratio]
\label{def:ldlr}
Let $\QQ = \cN(0,I_N)$ and let $\PP$ be a distribution on $\RR^N$ with finite moments of all orders. For an integer $d \ge 0$, let $\Pi_{\le d}: L^2(\QQ) \to L^2(\QQ)$ denote orthogonal projection in $L^2(\QQ)$ onto the subspace of degree-$\le d$ polynomials. The degree-$d$ likelihood ratio $L^\PP_{\le d}$ is traditionally defined as the best degree-$d$ approximation to the likelihood ratio $L^\PP := d\PP/d\QQ$ in the $L^2$ sense, namely
\begin{equation}\label{eq:ldlr-derivation}
L_{\le d}^\PP = \Pi_{\le d}(L^\PP) = \sum_{|\alpha| \le d} \langle L^\PP, h_\alpha \rangle_\QQ \, h_\alpha = \sum_{|\alpha| \le d} \EE_\PP[h_\alpha] h_\alpha.
\end{equation}
In some cases, $L^\PP$ may not exist or may not belong to $L^2(\QQ)$, making some steps above appear dubious. However, the right-hand side of~\eqref{eq:ldlr-derivation} remains well-defined because $\PP$ has finite moments, so we will take this as the formal definition:
\begin{equation}\label{eq:ldlr}
L_{\le d}^\PP := \sum_{|\alpha| \le d} \EE_\PP[h_\alpha] h_\alpha.
\end{equation}
Thus, $L^\PP_{\le d}$ is guaranteed to exist and belong to $L^2(\QQ)$, even if $L^\PP$ does not.
\end{definition}

The following basic properties are well known, but we include their proof for completeness.

\begin{proposition}[Basic properties of LDLR]
\label{prop:ldlr-facts}
In the setting of \Cref{def:ldlr}, the degree-$d$ likelihood ratio $L^\PP_{\le d}$, as defined in~\eqref{eq:ldlr}, satisfies:
\begin{enumerate}[label=(\roman*)]
    \item \label{item:ldlr-1} For any degree-$\le d$ polynomial $f: \RR^N \to \RR$, $\langle f,L^\PP_{\le d} \rangle_\QQ = \EE_\PP[f]$.
    \item \label{item:ldlr-2} The maximum value of the ratio
    \[ R^\PP(f) := \frac{\EE_\PP[f]}{\|f\|_\QQ} \]
    over (non-identically zero) degree-$\le d$ polynomials $f: \RR^N \to \RR$ is $\|L^\PP_{\le d}\|_\QQ$, achieved by $f = L^\PP_{\le d}$. The optimizer is unique up to scalar multiple.
    \item \label{item:ldlr-3} The norm of the LDLR can be computed as
    \[ \|L^\PP_{\le d}\|^2_\QQ = \sum_{|\alpha| \le d} (\EE_\PP[h_\alpha])^2. \]
    \item \label{item:ldlr-4} In particular, $\|L^\PP_{\le d}\|_\QQ \ge 1$.
\end{enumerate}
\end{proposition}

\begin{proof}
Write $f = \sum_{|\alpha| \le d} \hat{f}_\alpha h_\alpha$ for the Hermite expansion of a degree-$\le d$ polynomial $f$.

\ref{item:ldlr-1} Expanding in the orthonormal Hermite basis,
\[ \langle f,L^\PP_{\le d} \rangle_\QQ = \left\langle \sum_{|\alpha| \le d} \hat{f}_\alpha h_\alpha \, , \sum_{|\beta| \le d} \EE_\PP[h_\beta] h_\beta \right\rangle_{\!\!\QQ} = \sum_{|\alpha| \le d} \sum_{|\beta| \le d} \hat{f}_\alpha \, \EE_\PP[h_\beta] \langle h_\alpha,h_\beta \rangle_\QQ = \sum_{|\alpha| \le d} \hat{f}_\alpha \, \EE_\PP[h_\alpha] = \EE_\PP[f]. \]

\ref{item:ldlr-2} Again expand in the Hermite basis, and write $\hat{f} = (\hat{f}_\alpha)_{|\alpha| \le d}$ for the vector of coefficients. Note that
\begin{equation}\label{eq:parceval}
\|f\|_\QQ^2 = \sum_{|\alpha| \le d} \sum_{|\beta| \le d} \hat{f}_\alpha \hat{f}_\beta \, \EE_\QQ[h_\alpha h_\beta] = \sum_{|\alpha| \le d} \hat{f}_\alpha^2 = \|\hat{f}\|_2^2.
\end{equation}
Now
\[ R^\PP(f) = \frac{\sum_{|\alpha| \le d} \hat{f}_\alpha \, \EE_\PP[h_\alpha]}{\|\hat{f}\|_2} = \frac{\langle \hat{f},c \rangle}{\|\hat{f}\|_2}, \]
where $c = (c_\alpha)_{|\alpha| \le d}$ is the vector with entries $c_\alpha := \EE_\PP[h_\alpha]$. Thus, the ratio is maximized if and only if $\hat{f}$ is a (nonzero) scalar multiple of $c$. The maximum value of the ratio is $\|c\|_2$, which by~\eqref{eq:parceval} is equal to $\|L_{\le d}^\PP\|_\QQ$.

\ref{item:ldlr-3} This is the fact $\|L_{\le d}^\PP\|_\QQ = \|c\|_2$ referenced above.

\ref{item:ldlr-4} This comes from the $\alpha=0$ term in~\ref{item:ldlr-3} because $h_0$ is the constant function $1$.
\end{proof}

\section{General Reduction Proofs}
\label{sec:general}

We now present the details of the three modular components that underlie our proofs, discussed previously in \Cref{sec:pf-overview}. In this section we give non-asymptotic results, where bounds on e.g.\ failure probability hold exactly for every $n$, and asymptotic notation like $O(\cdot)$ is only used for runtimes.

\subsection{From ROC to ratio}
\label{sec:roc-to-ratio}

Here we will see that if you can beat the ROC curve that arises from thresholding the degree-$d$ LDLR, then you can achieve a value for the ratio $R^\PP$ from~\eqref{eq:ratio} that exceeds $\|L_{\le d}^\PP\|_\QQ$, meaning you achieve a larger value for $R^\PP$ than any degree-$d$ polynomial (see \Cref{prop:ldlr-facts}\ref{item:ldlr-2}).

This step is only used in one of our two main results: the computational Neyman--Pearson lemma but not the spiked Wigner application. The arguments in this section hold in high generality. We consider the setting of \Cref{def:ldlr} --- testing $\QQ = \cN(0,I_N)$ versus an arbitrary $\PP$ with finite moments --- since this is the context in which we have formally defined the LDLR.

\begin{definition}[LDLR threshold test]
\label{def:ldlr-threshold}
Consider the setting of \Cref{def:ldlr}. An \emph{LDLR threshold test of degree $d \ge 0$} is a test $B : \RR^N \to \{0,1\}$ of the following form: for some threshold $t \in \RR \cup \{\pm \infty\}$,
\[ B(y) = \begin{cases} 1 & \text{if } L_{\le d}^\PP(y) \ge t, \\ 0 & \text{if } L_{\le d}^\PP(y) < t. \end{cases} \]
\end{definition}

\begin{remark}
Assume the order-$\le d$ moments of $\PP$ and $\QQ$ do not all match exactly. This implies that $L_{\le d}^\PP$ is not a constant (degree-0) polynomial. Then for every $\alpha \in [0,1]$ there exists an LDLR threshold test with $\QQ(B=1) = \alpha$. To see this: each level set $\{y \,:\, L_{\le d}^\PP(y)=u\}$ is the zero set of a non-constant polynomial, hence has Gaussian measure zero, and therefore $L_{\le d}^\PP(Y)$ is atomless under $Y\sim\QQ$.
\end{remark}
    
\begin{lemma}[ROC to ratio]
\label{lem:direct-roc-to-ratio}
Let $\QQ = \cN(0,I_N)$ and let $\PP$ be a distribution on $\RR^N$ with finite moments of all orders. For some $d \ge 0$, suppose $A \randmap \RR^N\to\{0,1\}$ is a test that beats the ROC curve of the LDLR in the following sense: there is an LDLR threshold test $B$ of degree $d$ (\Cref{def:ldlr-threshold}) such that
\[ \QQ(A=1) = \QQ(B=1) \qquad \text{and} \qquad \PP(A=1) \ge \PP(B=1) + \eps, \]
for some $\eps>0$. Then the randomized score $F := L^\PP_{\le d} + \eps(A-B)$ satisfies
\[
\frac{\EE_{\PP}[F]}{\sqrt{\EE_\QQ[F^2]}} \ge \|L^\PP_{\le d}\|_\QQ + \frac{\eps^2}{3\|L^\PP_{\le d}\|_\QQ}.
\]
\end{lemma}

\begin{proof}
Write
\[ S := L^\PP_{\le d}. \]
Since $S$ is a degree-$\le d$ polynomial, we have by \Cref{prop:ldlr-facts}\ref{item:ldlr-1} that
\[ \EE_\PP[S] = \langle S,S \rangle_\QQ = \|S\|^2_\QQ. \]
From the definition of LDLR~\eqref{eq:ldlr} and basic properties of Hermite polynomials,
\[
\EE_\QQ[S] = 1.
\]
By assumption,
\[
\EE_\PP[A-B] = \PP(A=1)-\PP(B=1) \ge \eps.
\]
Let $t$ be the threshold associated with $B$ (see \Cref{def:ldlr-threshold}) and assume for now that $t$ is finite. Since $A,B\in\{0,1\}$ and $B$ is a threshold test at level $t$, we have pointwise
\[
(S-t)(A-B)\le 0.
\]
Indeed, on the event $\{S>t\}$ we have $B=1$ and hence $A-B\le 0$; on the event $\{S < t\}$ we have $B=0$ and hence $A-B\ge 0$; on the event $\{S = t\}$, the left-hand side vanishes. Therefore
\begin{align*}
\EE_\QQ[S(A-B)]
&= t\,\EE_\QQ[A-B]+\EE_\QQ[(S-t)(A-B)] \\
&\le t\bigl(\QQ(A=1)-\QQ(B=1)\bigr) \\
&= 0.
\end{align*}
In the case $t \in \{\pm \infty\}$, note that $\alpha \in \{0,1\}$, so $A=B$ with probability 1 under $\QQ$, and therefore $\EE_\QQ[S(A-B)] = 0$.

For $\eta>0$, define
\[
F_\eta:=S+\eta(A-B).
\]
We will bound $R^\PP(F_\eta)$. For the numerator,
\[
\EE_\PP[F_\eta]
=
\EE_\PP[S]+\eta\,\EE_\PP[A-B]
\ge \|S\|_\QQ^2+\eta\eps.
\]
For the denominator,
\begin{align*}
\EE_\QQ[F_\eta^2]
&= \EE_\QQ[S^2]+2\eta\,\EE_\QQ[S(A-B)]+\eta^2\EE_\QQ[(A-B)^2] \\
&\le \|S\|_\QQ^2+\eta^2,
\end{align*}
since $\EE_\QQ[S(A-B)]\le 0$ and $|A-B|\le 1$.
Hence
\[
\frac{\EE_{\PP}[F_\eta]}{\sqrt{\EE_\QQ[F_\eta^2]}}
\ge
\frac{\|S\|_\QQ^2+\eta\eps}{\sqrt{\|S\|_\QQ^2+\eta^2}}.
\]
Set $\eta = \eps$. To finish the proof, we need to show
\[ \sqrt{\|S\|_\QQ^2 + \eps^2} \ge \|S\|_\QQ + \frac{\eps^2}{3\|S\|_\QQ}. \]
To see this, square both sides and simplify to arrive at the equivalent condition $\epsilon^2 \le 3 \|S\|_\QQ^2$, which holds true: $\epsilon \le 1$ by assumption, and $\|S\|_\QQ \ge 1$ by \Cref{prop:ldlr-facts}\ref{item:ldlr-4}.
\end{proof}

\subsection{From ratio to one-sided test}
\label{sec:ratio-to-one-sided}

The remaining two components --- this subsection and the following one --- are used in both of our main results. Here we will see that if you have a randomized score function $F$ that beats the ratio achieved by the LDLR, then you can produce a one-sided distinguisher. We will need some regularity assumptions on $F$. To simultaneously handle the settings of our two main results, we will assume $F$ admits a decomposition $F = R+S$ where $R$ has bounded outputs and $S$ is a low-degree polynomial. For the computational Neyman--Pearson lemma, we use $F$ produced by \Cref{lem:direct-roc-to-ratio} in the previous subsection, which has this form by construction. For the spiked Wigner application, we will have a bounded $F$ produced by~\cite{precise-error}, and set $S = 0$.

For this component of the proof, we specialize to an additive Gaussian model. This step and the next one will each lose a small constant in the SNR, so we will be working with SNR values $0 < \lambda < \lambda'' < \lambda'$. When we refer to different planted distributions such as $\PP_\lambda$ and $\PP_{\lambda''}$, it is assumed that they all have the same distribution for the signal $X$ (see \Cref{def:agm}).

\begin{lemma}[Ratio to one-sided test]
\label{lem:ratio-to-one-sided}
Consider an additive Gaussian model (\Cref{def:agm}). Fix $0 < \lambda < \lambda''$ and write $\rho := \lambda/\lambda''$. Suppose the following hold for some $d \ge 0$.
\begin{enumerate}[label=(\roman*)]
    \item {\bf Beat the LDLR ratio:} There is a randomized function $F \randmap \RR^N \to \RR$ such that
    \begin{enumerate}[label=(\alph*)]
        \item \label{item:beat-ratio} $F$ beats the LDLR ratio, i.e.,
        \[
        \frac{\EE_{\PP_\lambda}[F]}{\sqrt{\EE_\QQ[F^2]}} \ge \|L_{\le d}^{\PP_\lambda}\|_{\QQ} + \Delta
        \]
        for some $\Delta \in (0,1]$,
        \item $F$ admits a decomposition $F = R + S$,
        \item $R$ is a randomized function whose outputs are always bounded by $|R(y)| \le M \|F\|_\QQ$ for some parameter $M \ge 1$,
        \item $S$ is a (deterministic) polynomial of degree $\le d$,
        \item $R$ and $S$ are both computable in time $T$.
    \end{enumerate}
    \item \label{item:ratio-of-norms} {\bf Ratio-of-norms bound:} There is a parameter $\tau \ge 1$ such that every degree-$\le d$ polynomial $p$ satisfies
    \[
    \|p\|_\PP \le \tau \|p\|_{\QQ} \qquad \text{for} \qquad \PP \in \{\PP_\lambda,\PP_{\lambda''}\}.
    \]
\end{enumerate}
Then there is a test $A'' \randmap \RR^N \to \{0,1\}$ computable in time $M^2 N^{O(d)} \cdot T$ such that
\begin{equation}\label{eq:one-sided-err-prob}
\QQ(A''=1) \le \frac{C \rho^{2d}}{\Delta^2}
\qquad\text{and}\qquad
\PP_{\lambda''}(A''=1) \ge \frac{\Delta^2}{C M^2 \tau^2}
\end{equation}
where $C > 0$ is an absolute constant. The algorithm for $A''$ requires access to the parameters $\rho,d,\Delta,M,\tau$.
\end{lemma}

\begin{remark}
Above, the expression $N^{O(d)}$ carries the implicit assumptions $N \ge 2$ and $d \ge 1$, and $O(\cdot)$ hides a multiplicative constant depending on $\rho$.
\end{remark}

\begin{proof}
We may assume
\begin{equation}\label{eq:tau-assum}
\frac{\rho^{2d}}{\Delta^2} \le \frac{\Delta^2}{M^2 \tau^2} \qquad \text{i.e.,} \qquad \tau \le \frac{\Delta^2}{M \rho^d},
\end{equation}
since otherwise the desired conclusion is achieved by the trivial test that outputs 1 with probability $\min\{1, \, \rho^{2d}/\Delta^2\}$ and 0 otherwise.

Define $f(y) = \EE[F(y)]$ where the expectation is over $F$'s internal randomness. By Jensen's inequality, $\EE_\QQ[f^2] \le \EE_\QQ[F^2]$, which is finite by assumption.

A function $f \in L^2(\QQ)$ can be decomposed into low- and high-degree parts as follows: recalling the notation from \Cref{def:ldlr},
\[ f_{\le d} := \Pi_{\le d} f, \qquad f_{> d} := f - f_{\le d}. \]
Note that $f_{\le d}$ is a polynomial of degree $\le d$, while $f_{> d}$ is orthogonal (with respect to $\langle \cdot , \cdot \rangle_\QQ$) to \emph{any} degree-$\le d$ polynomial. We emphasize that $f_{>d}$ has only high-degree terms ($|\alpha| > d$) in the Hermite basis, but may contain low-degree monomials when expanded in the monomial basis.

For our function $f = \EE[F]$ in question, we will use the shorthand $f = p+q$ where
\[
p := f_{\le d},
\qquad
q := f_{>d},
\qquad
\mu := T_\rho q,
\]
where $T_\rho$ denotes the Gaussian noise operator (see \Cref{sec:background-gaussian}) and, recall, $\rho := \lambda/\lambda''$. We will first show that thresholding $\mu$ at the value $\Delta \|F\|_\QQ / 3$ produces a test with the desired properties. The algorithm does not have direct access to $\mu$ but we will later show how to compute a suitable approximation to $\mu$.

\paragraph{Step 1: the null second moment is small.}
Expand the high-degree part, $q$, in the orthonormal Hermite basis as
\[
q = \sum_{|\alpha|>d} \hat q_\alpha h_\alpha.
\]
Recalling the action of $T_\rho$ on the Hermite basis~\eqref{eq:T-hermite} we have
\[
\mu = T_\rho q = \sum_{|\alpha|>d} \rho^{|\alpha|} \hat q_\alpha h_\alpha,
\]
and therefore
\[
\EE_\QQ[\mu^2]
= \sum_{|\alpha|>d} \rho^{2|\alpha|} \hat q_\alpha^2
\le \rho^{2d} \sum_{|\alpha|>d} \hat q_\alpha^2
\le \rho^{2d} \|f\|_{\QQ}^2
\le \rho^{2d} \|F\|_{\QQ}^2.
\]
Consequently, by Markov's inequality,
\begin{equation}\label{eq:mu-1}
\QQ\!\left(\mu \ge \frac{\Delta}{4}\|F\|_\QQ\right)
\le \frac{16}{\Delta^2 \|F\|_\QQ^2}\EE_\QQ[\mu^2]
\le \frac{16\rho^{2d}}{\Delta^2}.
\end{equation}

\paragraph{Step 2: the planted mean is large.}
Using $T_\rho \PP_{\lambda''} = \PP_\lambda$,
\[
\EE_{\PP_{\lambda''}}[\mu]
= \EE_{\PP_{\lambda''}}[T_\rho q]
= \EE_{\PP_\lambda}[q]
= \EE_{\PP_\lambda}[f] - \EE_{\PP_\lambda}[p].
\]
Now $p$ has degree $\le d$, so
\[
\EE_{\PP_\lambda}[p]
= \langle L_{\le d}^{\PP_\lambda}, p\rangle_{\QQ}
\le \|L_{\le d}^{\PP_\lambda}\|_{\QQ} \|p\|_{\QQ}
\le \|L_{\le d}^{\PP_\lambda}\|_{\QQ} \|F\|_{\QQ},
\]
where we have used \Cref{prop:ldlr-facts}\ref{item:ldlr-1} and the fact $\|p\|_\QQ \le \|f\|_\QQ \le \|F\|_\QQ$. From assumption~\ref{item:beat-ratio},
\[
\EE_{\PP_\lambda}[f] \ge \|L_{\le d}^{\PP_\lambda}\|_\QQ \|F\|_\QQ + \Delta \|F\|_\QQ.
\]
Combining the above yields
\[
\EE_{\PP_{\lambda''}}[\mu] \ge \Delta \|F\|_\QQ.
\]

\paragraph{Step 3: the planted second moment is controlled.}
By Jensen,
\begin{align*}
\EE_{\PP_{\lambda''}} [\mu^2] 
&= \EE_{Y \sim \PP_{\lambda''}} \! \left[\left(\EE_{G \sim \cN(0,I_N)} \, q(\rho Y + \sqrt{1-\rho^2} G)\right)^2 \right] \\
&\le \EE_{Y \sim \PP_{\lambda''}} \, \EE_{G \sim \cN(0,I_N)} [q(\rho Y + \sqrt{1-\rho^2} G)^2] \\
&= \EE_{\PP_\lambda}[q^2].
\end{align*}
From the decomposition $F = R+S$ we have $f = r+S$ where $r(y) = \EE[R(y)]$, and furthermore, $q = f_{>d} = r_{>d}$. Therefore
\begin{equation}\label{eq:q-r-high}
q = r - r_{\le d}
\end{equation}
and so
\[
q^2 \le 2r^2 + 2(r_{\le d})^2,
\]
which means
\[
\EE_{\PP_{\lambda''}}[\mu^2]
\le 2\EE_{\PP_\lambda}[r^2] + 2\EE_{\PP_\lambda}[(r_{\le d})^2].
\]
For the first term, we have $r^2 \le M^2 \|F\|_\QQ^2$ using the boundedness of $R$. For the second term, assumption~\ref{item:ratio-of-norms} gives
\[
\EE_{\PP_\lambda}[(r_{\le d})^2]
= \|r_{\le d}\|_{\PP_\lambda}^2
\le \tau^2 \|r_{\le d}\|_{\QQ}^2
\le \tau^2 \|r\|_{\QQ}^2
\le \tau^2 M^2 \|F\|_\QQ^2.
\]
Hence
\[
\EE_{\PP_{\lambda''}}[\mu^2] \le 2M^2(1+\tau^2) \|F\|_\QQ^2 \le 4 M^2 \tau^2 \|F\|_\QQ^2.
\]

\paragraph{Step 4: thresholding $\mu$ gives a one-sided test.}
Let $\nu := \max\{\mu,0\}$. Then
\[
\EE_{\PP_{\lambda''}}[\nu] \ge \EE_{\PP_{\lambda''}}[\mu] \ge \Delta \|F\|_\QQ,
\qquad
\EE_{\PP_{\lambda''}}[\nu^2] \le \EE_{\PP_{\lambda''}}[\mu^2] \le 4 M^2 \tau^2 \|F\|_\QQ^2.
\]
Applying the Paley--Zygmund inequality to $\nu$ yields
\begin{equation}\label{eq:mu-2}
\PP_{\lambda''}\!\left(\mu \ge \frac{\Delta}{2}\|F\|_\QQ\right)
\ge \PP_{\lambda''}\!\left(\nu \ge \frac12 \EE_{\PP_{\lambda''}}[\nu]\right)
\ge \frac{\EE_{\PP_{\lambda''}}[\nu]^2}{4\EE_{\PP_{\lambda''}}[\nu^2]}
\ge \frac{\Delta^2}{16 M^2 \tau^2}.
\end{equation}

\paragraph{Step 5: compute an approximation to $\mu$.}

At this point we have shown that thresholding $\mu$ at the value $\Delta \|F\|_\QQ/3$ produces a one-sided distinguisher with the desired behavior~\eqref{eq:one-sided-err-prob}, but the algorithm does not have direct access to $\mu$ or $\|F\|_\QQ$. We will now show how to approximately compute these. 

First we will show how to approximate $\mu(y)$ for a given $y$. Recall from~\eqref{eq:q-r-high} that $q = r - r_{\le d}$, and so
\begin{equation}\label{eq:r-high}
\mu(y) = T_\rho r(y) - T_\rho r_{\le d}(y).
\end{equation}
We will treat the two terms in~\eqref{eq:r-high} separately. For the first term,
\[
T_\rho r(y) = \EE_{G \sim \cN(0,I_N)}\!\left[r(\rho y + \sqrt{1-\rho^2}G)\right] = \EE\!\left[R(\rho y + \sqrt{1-\rho^2}G)\right],
\]
which can be estimated by averaging independent samples of the bounded random variable $R(\rho y + \sqrt{1-\rho^2}G)$ --- with $y$ fixed but $G,R$ random --- which lies in the interval $[-M \|F\|_\QQ, M \|F\|_\QQ]$. That is, draw independent Gaussian samples $G^{(1)},\ldots,G^{(m)} \sim \cN(0,I_N)$ and define the estimator
\[ \widehat{T_\rho r(y)} := \frac{1}{m} \sum_{j=1}^m R(\rho y + \sqrt{1-\rho^2}G^{(j)}), \]
where the internal randomness of $R$ is independent across terms. Fix a small absolute constant $c \in (0,1)$ to be chosen later. We will bound the failure probabilities of our approximation steps by $c \rho^{2d} / \Delta^2$ so that --- in light of~\eqref{eq:tau-assum} --- they do not interfere with the test's behavior in~\eqref{eq:one-sided-err-prob}. For any fixed $y$, apply Hoeffding's inequality and the fact $\exp(-x) \le 1/x$ for $x > 0$:
\begin{equation}\label{eq:m-cond-1}
\Pr\!\left(\left|\widehat{T_\rho r(y)}-T_\rho r(y)\right| \ge c\Delta\|F\|_\QQ\right) \le 2\exp\left(-\frac{c^2\Delta^2 m}{2M^2}\right) \le \frac{4M^2}{c^2\Delta^2m},
\end{equation}
which is $\le c \rho^{2d}/\Delta^2$ provided $m \ge \frac{4M^2}{c^3 \rho^{2d}}$. 

Now for the second term in~\eqref{eq:r-high},
\[ T_\rho r_{\le d}(y) = \sum_{|\alpha| \le d} \rho^{|\alpha|} \hat{r}_\alpha h_\alpha(y) = \sum_{|\alpha| \le d} \rho^{|\alpha|} h_\alpha(y) \, \EE[R(G)h_\alpha(G)], \]
where $G \sim \cN(0,I_N)$. Again, we can estimate this by averaging independent samples: draw independent Gaussians $G^{(j,\alpha)} \sim \cN(0,I_N)$ for $j \in [m]$ and $|\alpha| \le d$, and define
\[ \widehat{T_\rho r_{\le d}(y)} := \sum_{|\alpha| \le d} \rho^{|\alpha|} h_\alpha(y) \cdot \frac{1}{m} \sum_{j=1}^m R(G^{(j,\alpha)})h_\alpha(G^{(j,\alpha)}). \]
We will bound the variance of this estimator and then apply Chebyshev: for any fixed $y$,
\[ \Var\left(\widehat{T_\rho r_{\le d}(y)}\right) \le \frac{1}{m} \sum_{|\alpha| \le d} \rho^{2|\alpha|} h_\alpha(y)^2 \, \EE[R(G)^2 h_\alpha(G)^2]. \]
Now
\[ \EE[R(G)^2 h_\alpha(G)^2] \le M^2 \|F\|_\QQ^2 \, \EE[h_\alpha(G)^2] = M^2 \|F\|_\QQ^2, \]
so
\[ \Var\left(\widehat{T_\rho r_{\le d}(y)}\right) \le \frac{M^2 \|F\|_\QQ^2}{m} \, \sigma_y^2 \qquad \text{where} \qquad \sigma_y^2 := \sum_{|\alpha| \le d} h_\alpha(y)^2. \]
By Chebyshev, for any fixed $y$,
\begin{equation}\label{eq:est-h-1}
\Pr\left(\left|\widehat{T_\rho r_{\le d}(y)} - T_\rho r_{\le d}(y)\right| \ge c \Delta \|F\|_\QQ\right) \le \frac{M^2}{c^2 \Delta^2 m} \, \sigma_y^2.
\end{equation}
Now we bound $\sigma_Y^2$ with high probability over $Y \sim \QQ$ or $Y \sim \PP_{\lambda''}$. The number of $\alpha$ such that $|\alpha| \le d$ is at most $(N+1)^d$. We have
\[ \EE_{Y \sim \QQ}[\sigma_Y^2] = \sum_{|\alpha| \le d} \EE_\QQ[h_\alpha^2] = \sum_{|\alpha| \le d} 1 \le (N+1)^d, \]
and similarly, using assumption~\ref{item:ratio-of-norms},
\[ \EE_{Y \sim \PP_{\lambda''}}[\sigma_Y^2] = \sum_{|\alpha| \le d} \EE_{\PP_{\lambda''}}[h_\alpha^2] \le \sum_{|\alpha| \le d} \tau^2 \le (N+1)^d \tau^2. \]
By Markov, under either $Y \sim \QQ$ or $Y \sim \PP_{\lambda''}$,
\begin{equation}\label{eq:est-h-2}
\Pr(\sigma_Y^2 \ge t) \le \frac{(N+1)^d \tau^2}{t}.
\end{equation}
Now take a union bound over the previous two steps~\eqref{eq:est-h-1},\eqref{eq:est-h-2}, setting both failure probabilities to $c \rho^{2d}/\Delta^2$: under either $Y \sim \QQ$ or $Y \sim \PP_{\lambda''}$,
\begin{equation}\label{eq:m-cond-2}
\Pr\left(\left|\widehat{T_\rho r_{\le d}(Y)} - T_\rho r_{\le d}(Y)\right| \ge c \Delta \|F\|_\QQ\right) \le \frac{2c\rho^{2d}}{\Delta^2} \qquad \text{provided} \qquad m \ge \frac{M^2 \Delta^2 (N+1)^d \tau^2}{c^4 \rho^{4d}}.
\end{equation}

Finally, we will approximate $\|F\|_\QQ^2 = \EE_\QQ[F^2]$ so that we know where to threshold the (approximated) value of $\mu$. Again, draw independent Gaussian samples $G^{(1)},\ldots,G^{(m)} \sim \cN(0,I_N)$ and define
\[ \widehat{\|F\|_\QQ^2} := \frac{1}{m} \sum_{j=1}^m F(G^{(j)})^2. \]
Bound the variance:
\[ \Var\left(\widehat{\|F\|_\QQ^2}\right) \le \frac{1}{m} \EE_\QQ[F^4] = \frac{1}{m} \EE_\QQ[(R+S)^4] \le \frac{2^3}{m}\!\left(\EE_\QQ[R^4] + \EE_\QQ[S^4]\right) \le \frac{2^3}{m}\!\left(M^4 \|F\|_\QQ^4 + \EE_\QQ[S^4]\right). \]
Using \Cref{prop:gauss-bonami},
\[ \EE_\QQ[S^4] \le 9^d (\EE_\QQ[S^2])^2. \]
Furthermore,
\[ \EE_\QQ[S^2] = \EE_\QQ[(F-R)^2] \le 2 (\EE_\QQ[F^2] + \EE_\QQ[R^2]) \le 2(\|F\|_\QQ^2 + M^2 \|F\|_\QQ^2) \le 4M^2 \|F\|_\QQ^2, \]
and so,
\[ \Var\left(\widehat{\|F\|_\QQ^2}\right) \le C \cdot 9^d \cdot \frac{M^4 \|F\|_\QQ^4}{m} \]
for an absolute constant $C \ge 1$. By Chebyshev,
\begin{equation}\label{eq:m-cond-3}
\Pr\left(\left|\widehat{\|F\|_\QQ^2} - \|F\|_\QQ^2\right| \ge c \|F\|_\QQ^2\right) \le \frac{c \rho^{2d}}{\Delta^2} \qquad \text{provided} \qquad m \ge \frac{C M^4 \Delta^2 \, 9^d}{c^3 \rho^{2d}}.
\end{equation}

Putting it together, choose
\[ m \ge \frac{CM^4 \, 9^d (N+1)^d \tau^2}{c^4 \rho^{4d}} \]
to satisfy~\eqref{eq:m-cond-1},\eqref{eq:m-cond-2},\eqref{eq:m-cond-3}; we have used $\Delta \le 1$ here. The algorithm $A''$ receives a sample $Y$ and computes
\[ \widehat{\mu(Y)} := \widehat{T_\rho r(y)} - \widehat{T_\rho r_{\le d}(y)} \]
and thresholds it at the value $\Delta \widehat{\|F\|_\QQ^2}/3$. Combine the approximation results~\eqref{eq:m-cond-1},\eqref{eq:m-cond-2},\eqref{eq:m-cond-3}: under either $Y \sim \QQ$ or $Y \sim \PP_{\lambda''}$, with probability $\ge 1 - 4c\rho^{2d}/\Delta^2$,
\[ \left|\widehat{\mu(Y)} - \mu(Y)\right| \le 2c\Delta\|F\|_\QQ^2 \qquad \text{and} \qquad \left|\widehat{\|F\|_\QQ^2} - \|F\|_\QQ^2\right| \le c \|F\|_\QQ^2. \]
Choosing $c$ to be a sufficiently small absolute constant and recalling the behavior of $\mu$ from~\eqref{eq:mu-1},\eqref{eq:mu-2} along with~\eqref{eq:tau-assum}, the test $A''$ achieves the desired behavior~\eqref{eq:one-sided-err-prob}.

The runtime is dominated by drawing $m N^{O(d)}$ Gaussian samples and running the subroutines $R$ and/or $S$ for each sample. Therefore the runtime is
\[ m N^{O(d)} \cdot T = M^4 N^{O(d)} \tau^2 \cdot T \le M^2 N^{O(d)} \cdot T, \]
where we have used~\eqref{eq:tau-assum} and $\Delta \le 1$ in the last step.
\end{proof}

\subsection{From one-sided to two-sided test}
\label{sec:one-sided-to-two-sided}

We now show how to boost a one-sided distinguisher to a two-sided distinguisher, again in the context of additive Gaussian models. We will need to define a property called \emph{large-set expansion}, and to do so, we need the following definition.

\begin{definition}[Symmetry operator]
\label{def:sym-op}
Let $P = P_\omega$ be a randomized function $P \randmap \RR^N \to \RR^N$ with internal randomness $\omega \sim \Omega$ (see \Cref{sec:background-randf}). We call $P$ a \emph{symmetry operator} if for any $y \in \RR^N$, any fixed $\omega \in \Omega$, and a random $\omega' \sim \Omega$, we have that $P_{\omega'}(y)$ and $P_{\omega'}(P_\omega(y))$ have the same distribution.

A distribution $\QQ$ over $\RR^N$ is called \emph{$P$-invariant} if for $Y \sim \QQ$, we have that $Y$ and $P(Y)$ have the same distribution.

A function $g: \RR^N \to \RR$ is called \emph{$P$-invariant} if for every $y \in \RR^N$ and every $\omega \in \Omega$, we have $g(y) = g(P_\omega(y))$.
\end{definition}

An example to keep in mind, which we will use in the spiked Wigner setting, is the following: $\omega$ is a uniformly random permutation of $[n]$ and $P_\omega$ acts on an $n \times n$ matrix by permuting both the rows and columns by $\omega$. Another (trivial) example of a symmetry operator is the identity.

\begin{lemma}[One-sided to two-sided test]
\label{lem:one-sided-to-two-sided}
Consider an additive Gaussian model (\Cref{def:agm}). Fix $0 < \lambda'' < \lambda'$, and write $\rho := \lambda''/\lambda'$. Suppose there are parameters $\beta,\gamma,\eta,\delta \in (0,1)$ such that the following hold.
\begin{enumerate}[label=(\roman*)]
    \item \label{item:B-success} {\bf One-sided test:} There is a test $A''\randmap \RR^N \to \{0,1\}$ computable in time $T$ such that
    \[
    \QQ(A''=1) \le \delta
    \qquad and \qquad
    \PP_{\lambda''}(A''=1) \ge \beta.
    \]
    \item \label{item:sym-op}{\bf Symmetry:} There is a symmetry operator (\Cref{def:sym-op}) $P \randmap \RR^N \to \RR^N$ computable in time $T$ such that $\QQ$ and $\PP_{\lambda''}$ are both $P$-invariant.
    \item \label{item:rev-hyp} {\bf Large-set expansion:} For every $P$-invariant function $g:\RR^N \to [0,1]$ satisfying
    \[
    \EE_{\PP_{\lambda''}}[g] \ge \beta,
    \]
    we have
    \[
    \PP_{\lambda'}\!\left(T_\rho g \ge \gamma\right) \ge 1-\eta.
    \]
\end{enumerate}
Let $m \ge 1$ be an integer. Then there is a test $A'\randmap \RR^N \to \{0,1\}$ computable in time $O(mT)$ such that
\[
\QQ(A'=1) \le m\delta
\qquad \text{and} \qquad
\PP_{\lambda'}(A'=1) \ge 1-\eta-e^{-m\gamma}.
\]
The algorithm for $A'$ requires access to the parameters $\rho,m$.
\end{lemma}

In particular, if we are in an asymptotic regime where $m\delta = o(1)$, $\eta = o(1)$, and $m\gamma = \omega(1)$, then $A'$ is a two-sided distinguisher with success probability $1-o(1)$ under both $\QQ$ and $\PP_{\lambda'}$.

The presence of the symmetry operator makes it easier to verify large-set expansion. One could choose the symmetry operator $P = I_N$, and then every function is $P$-invariant. For our spiked Wigner application, we will take $P$ to be a random permutation as mentioned above, so that we only need to verify large-set expansion for permutation-invariant functions.

\begin{proof}
Let
\[
g(y) := \Pr\{A''(P(y))=1\},
\]
where the probability is over the internal randomness of $P$ and $A''$. Note that $g$ is $P$-invariant (in the sense of \Cref{def:sym-op}). Using the $P$-invariance of $\PP_{\lambda''}$ and the assumption~\ref{item:B-success} on $A''$,
\[
\EE_{\PP_{\lambda''}}[g] = \PP_{\lambda''}(A'' \circ P = 1) = \PP_{\lambda''}(A''=1) \ge \beta.
\]
We now define $A'$. On input $Y \in \RR^N$, generate independent Gaussian vectors $G^{(1)},\ldots,G^{(m)} \sim \cN(0,I_N)$ and form
\[
Y^{(i)} := \rho Y + \sqrt{1-\rho^2}G^{(i)}
\qquad \text{for } i=1,\ldots,m.
\]
Run $A'' \circ P$ independently on each $Y^{(i)}$, using fresh internal randomness each time, and output $1$ if and only if at least one of these $m$ runs outputs $1$. The runtime is $O(mT)$.

First suppose $Y \sim \QQ$. Then each $Y^{(i)}$ is distributed as $\QQ$, so by a union bound, and using the $P$-invariance of $\QQ$,
\[
\QQ(A'=1) \le \sum_{i=1}^m \QQ(A''(P(Y^{(i)}))=1) \le m\delta.
\]

Now suppose $Y \sim \PP_{\lambda'}$. Let
\[
E := \{T_\rho g(Y) \ge \gamma\}.
\]
By assumption~\ref{item:rev-hyp},
\[
\PP_{\lambda'}(E) \ge 1-\eta.
\]
Condition on a realization of $Y$ for which $E$ occurs. Since the Gaussian vectors $G^{(1)},\ldots,G^{(m)}$ and the internal randomness of the $m$ calls to $A'' \circ P$ are mutually independent, the $m$ trials are conditionally independent given $Y$. Moreover, for each $i$,
\[
\Pr\{A''(P(Y^{(i)}))=1 \mid Y\}
= \EE\!\left[g(Y^{(i)}) \mid Y\right]
= T_\rho g(Y)
\ge \gamma.
\]
Therefore, conditional on $Y \in E$,
\[
\Pr\{A'=0 \mid Y\} \le (1-\gamma)^m \le e^{-m\gamma}.
\]
Taking expectation over $Y$ gives
\[
\PP_{\lambda'}(A'=1)
\ge \PP_{\lambda'}(E)\bigl(1-e^{-m\gamma}\bigr)
\ge (1-\eta)\bigl(1-e^{-m\gamma}\bigr)
\ge 1-\eta-e^{-m\gamma},
\]
as claimed.

The runtime is dominated by making $m$ calls to the functions $P$ and $A''$, which gives the runtime $O(mT)$. An additional $O(mN)$ arithmetic operations and Gaussian samples are needed to prepare the inputs to those functions. Since $T = \Omega(N)$ --- the time required to write the output of $P$ --- this can be absorbed into $O(mT)$.
\end{proof}

\section{Computational Neyman--Pearson Lemma}
\label{sec:comp-np-formal}

We now state our computational Neyman--Pearson lemma, the formal version of \Cref{thm:comp-np-informal} from the Introduction.

\begin{theorem}\label{thm:comp-np}
Consider an additive Gaussian model (\Cref{def:agm}) in the asymptotic regime $n \to \infty$; the model is specified by $N = N_n \to \infty$ and $X = X_n$. Fix constants $0 < \lambda < \lambda'$, and set $\lambda'' := \sqrt{\lambda \lambda'}$ and $\rho := \lambda/\lambda'' = \lambda''/\lambda' \in (0,1)$. Fix an arbitrary sequence $d = d_n \to \infty$. Suppose the following hold for all sufficient large $n$.
\begin{itemize}
    \item {\bf Beat the ROC curve of LDLR:} Let $A \randmap \RR^N\to\{0,1\}$ be a test computable in time $T = T_n$. Let $B : \RR^N \to \{0,1\}$ be an LDLR threshold test of degree $d$ with threshold $t = t_n$ (\Cref{def:ldlr-threshold}) and suppose
    \[ \QQ(A=1) = \QQ(B=1) \qquad \text{and} \qquad \PP_\lambda(A=1) \ge \PP_\lambda(B=1) + \eps, \]
    for some constant $\eps > 0$.

    \item {\bf Ratio-of-norms bound:} There is a parameter $\tau = \tau_n \ge 1$ such that every degree-$\le d$ polynomial $p$ satisfies
    \[
    \|p\|_\PP \le \tau \|p\|_{\QQ} \qquad \text{for} \qquad \PP \in \{\PP_\lambda,\PP_{\lambda''}\}.
    \]

    \item {\bf Symmetry:} There is a symmetry operator (\Cref{def:sym-op}) $P \randmap \RR^N \to \RR^N$ computable in time $T = T_n$ such that $\QQ$ and $\PP_{\lambda''}$ are both $P$-invariant.
    
    \item {\bf Large-set expansion:} For some parameters $\beta_n,\gamma_n,\eta_n \in (0,1)$, the following holds. For every $P$-invariant function $g:\RR^N \to [0,1]$ (\Cref{def:sym-op}) satisfying
    \[
    \EE_{\PP_{\lambda''}}[g] \ge \beta,
    \]
    we have
    \[
    \PP_{\lambda'}\!\left(T_\rho g \ge \gamma\right) \ge 1-\eta.
    \]

    \item {\bf Scaling:} The parameters obey the asymptotics $\eps = \Omega(1)$, $\beta = o(1/\tau^2)$, $\gamma \le \beta$, $\gamma = \exp(-o(d))$, $\eta = o(1)$, and $\|L_{\le d}^{\PP_\lambda}\|_\QQ = O(1)$.
    
\end{itemize}
Then there is a test computable in time $N^{O(d)} \cdot T$ that strongly distinguishes (\Cref{def:strong-dist}) $\PP_{\lambda'}$ and $\QQ$. The algorithm requires access to the parameters $\lambda,\rho,d,t,\eps,\tau,\gamma$ as well as the order-$\le d$ moments of $X$, i.e., the values $\EE[X^\alpha]$ for $\alpha \in \NN^N$ with $|\alpha| \le d$.
\end{theorem}

We will see later that the spiked Wigner model satisfies the two main assumptions --- ratio-of-norms bound and large-set expansion --- with parameters that obey the required scalings; see \Cref{sec:ratio-of-norms,sec:lse}, as well as the proof of \Cref{prop:main-reduction-wigner} in \Cref{sec:wigner-pf-overview}.

One implication of our scaling assumptions is $\tau = \exp(o(d))$. This ensures that \Cref{lem:ratio-to-one-sided} will produce a nontrivial one-sided test, where the true positive rate is significantly larger than the false positive rate.

We have assumed the LDLR norm is bounded: $\|L_{\le d}^{\PP_\lambda}\|_\QQ = O(1)$. This is natural in light of the usual usage of the LDLR norm discussed in \Cref{sec:comp-np}. Namely, we require a low-degree lower bound for strong detection before we conclude anything more precise about weak detection. In the spiked Wigner model, the LDLR norm is bounded when $\lambda < 1$~\cite{ld-notes} (see~\cite{precise-error} for our exact version of the model).

\begin{remark}
We have assumed the algorithm has access to the parameter $t = t_n$. This is somewhat restrictive, as it may be nontrivial to compute the exact $t$ for which $\QQ(A=1) = \QQ(B=1)$ holds. One may prefer to think of this as a non-uniform algorithm (in the sense of circuit complexity) which receives the parameter $t$ as ``advice.'' It should also be possible to approximate $t$ to high accuracy by taking samples from $\QQ$, but we have opted not to analyze this in favor of a simpler proof.
\end{remark}

The proof is essentially a direct combination of the three general steps from \Cref{sec:general}.

\begin{proof}
Assume $\eps \le 1/2$ without loss of generality. By \Cref{lem:direct-roc-to-ratio}, the randomized score $F := L^{\PP_\lambda}_{\le d} + \eps(A-B)$ satisfies
\[
R^{\PP_\lambda}(F) \ge \|L^{\PP_\lambda}_{\le d}\|_\QQ + \frac{\eps^2}{3\|L^{\PP_\lambda}_{\le d}\|_\QQ}.
\]

Next we apply \Cref{lem:ratio-to-one-sided} with
\[ \Delta = \min\left\{1,\,\frac{\eps^2}{3\|L^{\PP_\lambda}_{\le d}\|_\QQ}\right\} \]
and $R = \eps(A-B)$, $S = L_{\le d}^{\PP_\lambda}$. To choose the parameter $M$, note that $|R(y)| \le \eps \le 1/2$, so we need a lower bound on $\|F\|_\QQ$. Write
\[ \|F\|_\QQ \ge \|S\|_\QQ - \|R\|_\QQ \ge 1 - \eps \ge 1/2, \]
where we have used \Cref{prop:ldlr-facts}\ref{item:ldlr-4}. Therefore we can take $M = 1$.

We now address how to compute $R$ and $S$. Since we assume access to the moments of $X$, the LDLR $L_{\le d}^{\PP_\lambda}$ can be computed from its definition~\eqref{eq:ldlr} in time $N^{O(d)}$, as \Cref{prop:hermite-shift} implies
\[ \EE_{\PP_\lambda}[h_\alpha] = \EE[h_\alpha(\lambda X + Z)] = \frac{\lambda^{|\alpha|}}{\sqrt{\alpha!}} \EE[X^\alpha]. \]
Since we assume access to the parameters $t$ and $\eps$, this means $R$ can be computed in time $N^{O(d)}+T$ and $S$ can be computed in time $N^{O(d)}$. Similarly, $\|L_{\le d}^{\PP_\lambda}\|_\QQ$ --- and thus $\Delta$ --- can be computed in time $N^{O(d)}$ using the formula in \Cref{prop:ldlr-facts}\ref{item:ldlr-3}.

From the conclusion of \Cref{lem:ratio-to-one-sided} we now have a test $A''$, computable in time $N^{O(d)} \cdot T$, such that
\[ \QQ(A''=1) \le O(\rho^{2d}) \qquad \text{and} \qquad \PP_{\lambda''}(A''=1) \ge \Omega(1/\tau^2), \]
where we have used $\Delta = \Omega(1)$ due to the assumptions $\eps = \Omega(1)$ and $\|L_{\le d}^{\PP_\lambda}\|_\QQ = O(1)$.

Next we apply \Cref{lem:one-sided-to-two-sided} to produce a test $A'$. The above behavior of $A''$ implies we can set $\delta = C\rho^{2d}$ for a constant $C > 0$ (depending on $\Delta$). Also, our assumption $\beta = o(1/\tau^2)$ implies $\PP_{\lambda''}(A''=1) \ge \beta$ for sufficiently large $n$, as required. Choose
\[ m = \lceil 1/\sqrt{\gamma\delta} \rceil. \]
Note that $\gamma\delta = o(1)$ because $\delta = \exp(-\Theta(d))$ from above, and $\gamma \le \beta = o(1/\tau^2) \le o(1)$. Therefore $m = \Theta(1/\sqrt{\gamma\delta})$. Also, using the assumption $\gamma = \exp(-o(d))$, we have $\delta = o(\gamma)$. Therefore $m\delta = o(1)$ and $m\gamma = \omega(1)$. Together with the assumption $\eta = o(1)$, the conclusion of \Cref{lem:one-sided-to-two-sided} gives $\QQ(A'=1) = o(1)$ and $\PP_{\lambda'}(A'=1) = 1-o(1)$ as desired. The runtime of $A'$ is $m N^{O(d)} \cdot T$, where $m = \exp(O(d)) \le N^{O(d)}$.
\end{proof}

\section{Spiked Wigner Application}
\label{sec:wigner-application}

This section is devoted to the proof of \Cref{thm:main-wig}, with some ingredients deferred to \Cref{app:ratio-norms} and \Cref{app:lse}. Note that the statement of \Cref{thm:main-wig} only becomes weaker for larger $d$, so we are free to make assumptions such as $d = o(\log n / \log \log n)$ without loss of generality.

\subsection{Equivalent additive Gaussian formulation}
\label{sec:equiv-gaussian}

It will be convenient to work with an equivalent variation of the spiked Wigner model, where redundant information is removed and the diagonal entries are rescaled so that the null distribution becomes a standard Gaussian vector.

\begin{definition}[Spiked Wigner model --- additive Gaussian formulation]
\label{def:wigner-additive}
The observations will be indexed by
\[
E := \{(i,j) : 1 \le i < j \le n\} \sqcup [n],
\]
corresponding to the upper-triangular entries of an $n\times n$ matrix, where $i \in [n]$ corresponds to the diagonal entry $(i,i)$. Let $N = |E| = n + \binom{n}{2}$, and identify $\RR^N$ with $\RR^E$. Consider the linear map $\Phi:\RR^{n\times n}\to \RR^{N}$ defined by
\[\Phi(Y) = \tri(\phi(Y)),\]
where $\phi:\RR^{n\times n}\to \RR^{n\times n}$ has the effect of scaling each diagonal entry by a factor of $1/\sqrt{2}$, and $\tri:\RR^{n\times n}\to \RR^{N}$ extracts the upper-triangular part of its input, in the sense that $\tri(Y)_{ij} = Y_{ij}$ for $i < j$ and $\tri(Y)_i = Y_{ii}$ for $i \in [n]$.  Now consider testing between the following two distributions:
\begin{itemize}
    \item Under $\QQ = \QQ_n$, $Y \in \RR^N$ is an i.i.d.\ standard Gaussian vector $Y \sim \cN(0,I_N)$.
    \item Under $\PP_\lambda = \PP_{\lambda,n}$, $Y = \lambda X + Z$ where $Z \sim \cN(0,I_N)$ and
    \begin{equation}\label{eq:X}
    X = \mathbf{1}\{x\neq 0\}\cdot \sqrt{n}\cdot \frac{\Phi(xx^\top)}{\|x\|^2}\quad\text{where }x\sim \pi^{\otimes n}.
    \end{equation}
\end{itemize}
As usual, we will be interested in the regime $n \to \infty$ with $\lambda$ and $\pi$ held fixed.
\end{definition}

This is an additive Gaussian model in the sense of \Cref{def:agm}. Furthermore, this model is equivalent to the original spiked Wigner model (\Cref{def:wig-intro}) via the linear transformation $\Phi$, which is invertible on the domain of symmetric $n \times n$ matrices. As a result, it will suffice to consider this new model for our proof of \Cref{thm:main-wig}, allowing us to apply the general machinery for additive Gaussian models from \Cref{sec:general}.

\subsection{Proof overview}
\label{sec:wigner-pf-overview}

Our main objective will be to establish the following by combining the ingredients from \Cref{sec:ratio-to-one-sided,sec:one-sided-to-two-sided}.

\begin{proposition}[Main reduction for spiked Wigner: from ratio to strong detection]
\label{prop:main-reduction-wigner}
Consider the spiked Wigner model as formulated in \Cref{def:wigner-additive}. Fix constants $0 < \lambda < \lambda' < 1$ and fix a prior $\pi$ that is mean-zero, variance-one, and supported on a finite subset of $\RR$. Fix an arbitrary sequence $d = d_n \to \infty$. Suppose we have a randomized function $F \randmap \RR^N \to \RR$ that is computable in time $T = T_n$ and always bounded by $F(y) \in [1/B,B]$ for a constant $B \ge 1$. For all sufficiently large $n$, suppose further that $F$ beats the LDLR ratio, i.e.,
\[ \frac{\EE_{\PP_\lambda}[F]}{\sqrt{\EE_\QQ[F^2]}} \ge \|L_{\le d}^{\PP_\lambda}\|_\QQ + \Delta \]
for a constant $\Delta > 0$. Then there is a test computable in time $n^{O(d)} \cdot T$ that strongly distinguishes $\PP_{\lambda'}$ and $\QQ$. The algorithm requires access to the parameters $\lambda, \lambda', \pi, d, B, \Delta$.
\end{proposition}

Before proving this, we show how it implies our main result using machinery from~\cite{precise-error}.

\begin{proof}[Proof of \Cref{thm:main-wig}]
This is similar to the proof of~\cite[Corollary~2.22]{precise-error} (using the numbering in arXiv v3). Without loss of generality, assume $d = o(n / \log n)$. By assumption, we have a test $A$ with runtime $T$ that achieves some point $(\alpha,\beta)$ above the ROC curve $\phi_\lambda$. There is also a polynomial-time algorithm, namely LSS, that asymptotically achieves any given point on the curve $\phi_\lambda$ (this is stated as Theorem~2.10 in~\cite{precise-error} but follows from previous work~\cite{sk-ferro,weak-wigner}). With these two ingredients in hand, we can invoke~\cite[Proposition~2.20]{precise-error} (with properties of $\phi_\lambda$ provided by~\cite[Lemmas~2.12 and~2.19]{precise-error}) to obtain a function $F \randmap \RR^N \to \RR$ that achieves the ratio
\[ \frac{\EE_{\PP_\lambda}[F]}{\sqrt{\EE_\QQ[F^2]}} \ge (1-\lambda^2)^{-1/4} + \Omega(1). \]
The function $F$ has runtime $O(T) + n^{O(1)}$, requires knowledge of $\lambda,\alpha,\beta$, and may be randomized because it uses $A$ as a subroutine. By~\cite[Theorem~2.13]{precise-error} and the assumption $d = o(n / \log n)$,
\[ \|L_{\le d}^{\PP_\lambda}\|_\QQ = (1-\lambda^2)^{-1/4} + o(1), \]
so $F$ beats the LDLR ratio as required for \Cref{prop:main-reduction-wigner}. Furthermore, $F$ outputs values from a finite list of positive constants (see~\cite[Remark~2.21]{precise-error}), so $F(y) \in [1/B,B]$ for some constant $B \ge 1$. Now apply \Cref{prop:main-reduction-wigner} to obtain a test computable in time $n^{O(d)} \cdot T$ that strongly distinguishes $\PP_{\lambda'}$ and $\QQ$. The parameters $B,\Delta$ for \Cref{prop:main-reduction-wigner} are known to the algorithm, as they are constants that only depend on the known parameters $\lambda,\alpha,\beta$.
\end{proof}

Now we give the proof of \Cref{prop:main-reduction-wigner}, with some technical ingredients deferred to later sections, namely the verification of two properties for spiked Wigner: ratio-of-norms bound (\Cref{sec:ratio-of-norms}) and large-set expansion (\Cref{sec:lse}).

\begin{proof}[Proof of \Cref{prop:main-reduction-wigner}]
Without loss of generality, assume $d = o(\log n / \log \log n)$. Let $\lambda'' := \sqrt{\lambda \lambda'}$ and $\rho := \lambda/\lambda'' = \lambda''/\lambda' \in (0,1)$. First apply \Cref{lem:ratio-to-one-sided} with $R = F$ and $S = 0$. Due to the boundedness assumption on $F$, we can take the parameter $M \ge 1$ to be a constant depending on $B$. The ratio-of-norms bound comes from \Cref{lem:ratio-norms} below (noting that $\pi$ is subgaussian because it has finite support), with $\tau = d^{O(\sqrt{d})}$. The conclusion of \Cref{lem:ratio-to-one-sided} gives a test $A''$, computable in time $N^{O(d)} \cdot T$, such that
\[ \QQ(A''=1) \le O(\rho^{2d}) = \exp(-\Omega(d)) \qquad \text{and} \qquad \PP_{\lambda''}(A''=1) \ge \Omega(1/\tau^2) = d^{-O(\sqrt{d})}. \]

This allows us to apply \Cref{lem:one-sided-to-two-sided} with certain parameters $\delta = \exp(-\Theta(d))$ and $\beta = d^{-\Theta(\sqrt{d})}$. Here we use the symmetry operator $P$ that applies a uniformly random permutation to $[n]$ (see \Cref{def:perm-inv}). Large-set expansion comes from \Cref{lem:lse} below, which provides the parameters $\eta = \beta$ and $\gamma = \beta^c$ for a constant $c \ge 1$; so $\eta,\gamma$ are both $d^{-\Theta(\sqrt{d})}$. Note that $\gamma \le \beta = \eta \le o(1)$ and $\delta = o(\gamma)$, so similar to the proof of \Cref{thm:comp-np} we can choose $m = \lceil 1/\sqrt{\gamma\delta} \rceil$ and conclude that the resulting test $A'$ achieves $\QQ(A'=1)=o(1)$ and $\PP_{\lambda'}(A'=1)=1-o(1)$ as desired. Since $m = \exp(O(d))$, the runtime is $N^{O(d)} \cdot T$.
\end{proof}

\subsection{Ratio of norms}
\label{sec:ratio-of-norms}

The following ratio-of-norms bound was used in the proof of \Cref{prop:main-reduction-wigner} above. This holds for the class of subgaussian priors (see \Cref{sec:subg}), which includes all finitely-supported priors.

\begin{lemma}[Ratio-of-norms bound for spiked Wigner]
\label{lem:ratio-norms}
Consider the spiked Wigner model as formulated in \Cref{def:wigner-additive}. Fix $0 \le \lambda\le 1$ and fix a prior $\pi$ that is mean-zero, variance-one, and subgaussian (\Cref{def:subg}). Suppose $d = o(\log n/\log\log n)$. Then for every degree-$\le d$ polynomial $p:\RR^N \to \RR$,
\[
\|p\|_{\PP_\lambda} \le d^{O(\sqrt{d})}\cdot \|p\|_{\QQ}.
\]
\end{lemma}

The proof is deferred to \Cref{app:ratio-norms} but we give a brief summary here. Consider the Hermite expansion of an arbitrary degree-$\le d$ polynomial: $p = \sum_{|\alpha| \le d} \hat{p}_\alpha h_\alpha$. Note that
\[ \frac{\|p\|^2_{\PP_\lambda}}{\|p\|^2_\QQ} = \frac{\hat{p}^\top M \hat{p}}{\|\hat{p}\|_2^2} \]
where $M = (M_{\alpha\beta})_{|\alpha|,|\beta| \le d}$ is the ``Hermite Gram matrix'' with entries
\[ M_{\alpha\beta} := \EE_{\PP_\lambda}[h_\alpha h_\beta]. \]
Therefore, the maximum value of $\|p\|^2_{\PP_\lambda}/\|p\|^2_\QQ$ is the spectral norm of $M$. One tractable upper bound on the spectral norm is the maximum $\ell_1$-norm of any row of $M$. This bound is not quite good enough for our purposes, but a modification of this strategy works: we will first partition the terms of $p$ based on their structure, and then apply the $\ell_1$ bound separately within each group. The remaining details are deferred to \Cref{app:ratio-norms}.

We make some remarks on the conclusion of \Cref{lem:ratio-norms}. As mentioned earlier, it is important for our downstream application that the ``blow-up factor'' $d^{O(\sqrt{d})}$ is subexponential in $d$, i.e.\ $\exp(o(d))$, since this must ``compete'' against $\rho^{2d}$ in~\eqref{eq:one-sided-err-prob}. The assumption $\lambda \le 1$ is essential here: when $\lambda > 1$, the ratio $\|p\|_{\PP_\lambda}/\|p\|_\QQ$ can be exponentially large in $d$, i.e., $\Omega(\lambda^d)$. For instance, this is achieved by the ``$d$-cycle polynomial''
\[ p(Y) = \sum_{\substack{i_1,i_2,\ldots,i_d \in [n] \\ \mathrm{distinct}}} Y_{i_1,i_2} Y_{i_2,i_3} \cdots Y_{i_d,i_1}. \]

In \Cref{lem:ratio-norms} we have assumed $d = o(\log n / \log \log n)$ for convenience, as we are primarily interested in the case where $d$ grows arbitrarily slowly with $n$. It is likely that this requirement on $d$ can be relaxed --- perhaps to $d = o(n)$ --- with a more careful proof, but we have not attempted this here.

\begin{remark}[The exponent $\sqrt{d}$]
The exponent $\sqrt{d}$ that we obtain in the ``blow-up factor'' is essentially sharp, and its origin can be understood by comparison to a scalar Gaussian problem where $\QQ$ is $\cN(0,1)$ and $\PP$ is $\cN(\mu,1)$ for a constant $\mu > 0$. In this scalar problem, the maximum ratio $\|p\|_\PP/\|p\|_\QQ$ over degree-$d$ polynomials scales as $\exp(\Theta(\sqrt{d}))$, as we will justify shortly. The spiked Wigner problem ``contains'' this scalar problem by considering the trace of the observed matrix, so certainly the Wigner problem cannot have a smaller blow-up factor than the scalar problem. The fact that both have similar blow-up factors (when $\lambda \le 1$) reflects the intuition that, to low-degree polynomials, the Wigner problem ``looks like'' the scalar one. This intuition is consistent with the fact that LSS, the conjectured-optimal algorithm for the Wigner problem, behaves like the scalar problem (when $\lambda \le 1$).

We now sketch the analysis of the scalar Gaussian problem claimed above. Expand an arbitrary $p$ in the Hermite basis: $p(z) = \sum_{k=0}^d \hat{p}_k h_k(z)$. Now $\|p\|_\QQ = \|\hat{p}\|_2$ and
\[ \|p\|_\PP \le \sum_{k=0}^d \hat{p}_k \|h_k\|_\PP \le \|\hat{p}_k\|_2 \cdot \sqrt{\sum_{k=0}^d \|h_k\|_\PP^2}. \]
Up to a factor of $\sqrt{d+1}$, the blow-up factor that we seek is the maximum of the values $\|h_k\|_\PP$. Write $y = \mu + z$ with $z \sim \cN(0,1)$ and apply \Cref{prop:hermite-shift} along with Hermite orthonormality to arrive at
\[ \|h_k\|_\PP^2 = \EE[h_k(\mu+z)^2] = \sum_{\ell=0}^k \frac{\ell!}{k!} \binom{k}{\ell}^2 \mu^{2(k-\ell)}. \]
This has value $\exp(\Theta(\sqrt{k}))$, with the dominant contribution coming from $\ell \approx k - \mu \sqrt{k}$.
\end{remark}

\subsection{Large-set expansion}
\label{sec:lse}

We will now state the large-set expansion property for spiked Wigner, which is the final ingredient that was used in the proof of \Cref{prop:main-reduction-wigner} above. This will apply to permutation-invariant functions, which we first define.

\begin{definition}[Permutation invariance]
\label{def:perm-inv}
Consider the spiked Wigner model as formulated in \Cref{def:wigner-additive}, where the observations $Y \in \RR^N$ are indexed by either $i$ or $(i,j)$ with $i < j$. A permutation $\sigma \in S_n$ acts on $y \in \RR^N$ by permuting the indices $1,\ldots,n$, namely $(\sigma \cdot y)_i = y_{\sigma^{-1}(i)}$ and $(\sigma \cdot y)_{i,j} = y_{\sigma^{-1}(i),\sigma^{-1}(j)}$ (with the convention $y_{i,j} := y_{j,i}$ when $i > j$).

A function $g: \RR^N \to \RR$ is called \emph{permutation-invariant} if $g(y) = g(\sigma \cdot y)$ for every $y \in \RR^N$ and every $\sigma \in S_n$.
\end{definition}

\begin{lemma}[Large-set expansion for spiked Wigner]
\label{lem:lse}
Consider the spiked Wigner model as formulated in \Cref{def:wigner-additive}. Fix $0 < \lambda < \lambda'$ and write $\rho := \lambda/\lambda'$. Fix a prior $\pi$ that is mean-zero, variance-one, and supported on a finite set. Then there exist constants $c,\kappa,\eps_0>0$ depending only on $\lambda,\lambda',\pi$, such that the following holds for all sufficiently large $n$. For any permutation-invariant $g:\RR^N\to[0,1]$ with $\EE_{\PP_{\lambda}}[g]\ge \eps$ for some $2^{-\kappa n}\le \eps \le \eps_0$, we have
\[
\PP_{\lambda'}\!\left(T_\rho g\ge \eps^{c}\right)\ge 1-\eps.
\]
\end{lemma}

\noindent The proof is deferred to \Cref{app:lse} but we give a brief summary here. As we will see in \Cref{sec:fixed-signal-lse}, this expansion property is already known to hold for the null distribution $\cN(0,I_N)$ and therefore holds for a variant of the spiked Wigner model where the spike $x$ is fixed to one deterministic vector. Due to our assumption of permutation invariance, expansion also follows for a variant of the model where the empirical distribution of $x$ is fixed but the entries are permuted at random. The main technical challenge is to handle the variability in the empirical distribution. To do this, we show that the spiked distributions induced by different ``typical'' empirical distributions are statistically ``close'' to each other. This involves computing a particular chi-squared divergence between one distribution with a fixed spike, and one with a random spike from a fixed empirical distribution. The details for how to sample the random spike are somewhat subtle, and we refer to \Cref{app:lse} for the full argument. We note that this argument is the only part of our proof where we use the assumption that $\pi$ has finite support.

\appendix

\section{Ratio of Norms for Spiked Wigner}
\label{app:ratio-norms}

This section is devoted to the proof of Lemma~\ref{lem:ratio-norms}.

\subsection{Subgaussian random variables}
\label{sec:subg}

In this section we will assume $\pi$ is \emph{subgaussian}, which has various equivalent definitions (e.g.~\cite[Proposition~2.6.1]{vershynin-hdp}). We will use the following definition in terms of moments.

\begin{definition}\label{def:subg}
A real-valued random variable $X$ is called \emph{subgaussian} if there exists $K > 0$ such that
\[
\EE[|X|^p] \le (K \sqrt{p})^p \qquad \text{for all } p \ge 1.
\]
\end{definition}

We will need the following tail bound for a sum of squares of subgaussian random variables. This can be found in~\cite{vershynin-hdp}, specifically the Bernstein inequality for subexponential random variables (Theorem~2.9.1) together with the fact that the square of a subgaussian random variable is subexponential (Lemma~2.8.5).

\begin{proposition}
\label{prop:subgaussian-square-lower-tail}
Let $u_1,\ldots,u_m$ be i.i.d.\ from a subgaussian law $\pi$ with mean 0 and variance 1. Then there exists a constant $c=c(\pi)>0$ such that for every $0 \le t \le m$,
\[
\Pr\!\left[\sum_{i=1}^m u_i^2 \le m-t\right]
\le \exp\!\left(-c \, \frac{t^2}{m}\right).
\]
\end{proposition}

\subsection{Proof overview}

The high-level structure of the proof is as follows. We will decompose $p = \sum_{\sigma}p_\sigma$ into a small number of terms, where $p_\sigma$ contains all terms of $p$ in the Hermite basis corresponding to a particular ``signature'' $\sigma$. Since the number of terms in the above decomposition is small enough, it will suffice to prove the ratio-of-norms inequality for each individual $p_\sigma$. This is equivalent to bounding the spectral norm of a certain Gram matrix $M^\sigma$, which we will control using a row-sum bound.

Let us now formally define the decomposition. We identify a multi-index $\alpha \in \NN^E$ with the undirected multigraph on vertex set $[n]$ having $\alpha_{ij}$ copies of the edge $\{i,j\}$ for all distinct pairs of vertices $i,j$, and $\alpha_i$ copies of the self-loop at $i$ for each vertex $i$.

For such a multigraph $\alpha$, we use the following notation.
\begin{itemize}
    \item $V(\alpha) \subseteq [n]$ is the set of non-isolated vertices.
    \item $C^\alpha : \NN_{\geq 1} \to \NN$ records the cycle components of $\alpha$: for each $r \ge 1$, $C^\alpha(r)$ is the number of connected components of $\alpha$ that are simple cycles on $r$ vertices. Here a self-loop counts as a simple $1$-cycle and a double edge counts as a simple $2$-cycle.
    \item The \emph{signature} of $\alpha$ is the $3$-tuple
    \[
    s(\alpha) := \bigl(|\alpha|,\ |V(\alpha)|,\ C^\alpha\bigr).
    \]
    \item $\core(\alpha)$ is the multigraph obtained from $\alpha$ after deleting every connected component that is a simple cycle.
\end{itemize}
Define the set of relevant signatures
\[
\Sigma_d := \{\sigma : \sigma = s(\alpha) \text{ for some } \alpha \in \NN^E \text{ with } |\alpha| \le d\}.
\]

\begin{lemma}\label{lem:ratio-signature-count}
We have $|\Sigma_d| \le \exp(O(\sqrt{d}))$.
\end{lemma}
\begin{proof}
The first two coordinates of a signature contribute at most $(d+1)\cdot (2d+1)$ possibilities. It therefore suffices to bound the number of possible cycle-count functions $C:\NN_{\geq 1}\to \NN$ satisfying
\[
\sum_{r \ge 1} r\cdot C(r) \le d.
\]
Such a function is exactly an integer partition of a number that is at most $d$, so the number of possibilities is
\[
    \sum_{t=0}^d p(t) \le \exp(O(\sqrt{d})),
\]
where $p(t)$ denotes the partition function, and the inequality follows from the classical Hardy--Ramanujan asymptotic formula~\cite{hardy-ramanujan-partitions}.
\end{proof}

Let $p : \RR^N \to \RR$ be a degree-$\le d$ polynomial, and write its expansion in the Hermite basis as
\[
p = \sum_{|\alpha| \le d} \hat p_\alpha \, h_\alpha.
\]
For each $\sigma \in \Sigma_d$, define
\[
p_\sigma := \sum_{s(\alpha)=\sigma} \hat p_\alpha \, h_\alpha,
\]
and note that $p = \sum_{\sigma \in \Sigma_d} p_\sigma$. For each $\sigma\in \Sigma_d$, we will prove
\begin{equation}\label{eqn:norm-bound-sig}
	\|p_\sigma\|_{\PP_\lambda}\leq d^{O(\sqrt{d})}\cdot \|p_\sigma\|_{\QQ}.
\end{equation}
By the following sequence of inequalities, this implies the desired conclusion:
\begin{align*}
\|p\|_{\PP_\lambda}
&\le \sum_{\sigma \in \Sigma_d} \|p_\sigma\|_{\PP_\lambda}
\tag{Triangle inequality} \\
&\le d^{O(\sqrt{d})}\cdot\sum_{\sigma \in \Sigma_d}\|p_\sigma\|_{\QQ}
\tag*{using \eqref{eqn:norm-bound-sig}} \\
&\le d^{O(\sqrt{d})}\cdot\sqrt{|\Sigma_d|}\cdot\left(\sum_{\sigma \in \Sigma_d}\|p_\sigma\|_{\QQ}^2\right)^{1/2}
\tag{Cauchy--Schwarz} \\
&= d^{O(\sqrt{d})}\cdot\sqrt{|\Sigma_d|}\cdot\|p\|_{\QQ}
\tag{orthogonality of Hermite polynomials under $\QQ$} \\
&\le d^{O(\sqrt{d})}\cdot\|p\|_{\QQ}.
\tag{\Cref{lem:ratio-signature-count}}
\end{align*}
It remains to show $\eqref{eqn:norm-bound-sig}$. For $\sigma\in \Sigma_d$, define the Gram matrix
\[
M^\sigma_{\alpha,\beta} := \EE_{\PP_\lambda}[h_\alpha(Y)h_\beta(Y)],
\qquad \alpha,\beta \in s^{-1}(\sigma).
\]
For intuition, $\|p_\sigma\|^2_{\PP_\lambda} = \hat{p}_\sigma^\top M^\sigma \hat{p}_\sigma$ and $\|p_\sigma\|^2_\QQ = \|\hat{p}_\sigma\|^2$ where $\hat{p}_\sigma$ is the vector of Hermite coefficients of $p_\sigma$, so our goal is to bound the spectral norm of $M^\sigma$. We will do this using the following row-sum bound.

\begin{lemma}\label{lem:row-sum-bound}
Consider the setting of \Cref{lem:ratio-norms}. For every $\sigma \in \Sigma_d$ and every $\alpha \in s^{-1}(\sigma)$,
\[
\sum_{\beta \in s^{-1}(\sigma)} \left| \EE_{\PP_\lambda}[h_\alpha(Y)h_\beta(Y)] \right|
\le d^{O(\sqrt{d})}.
\]
\end{lemma}

Before proving this lemma, let us see how it implies our goal~\eqref{eqn:norm-bound-sig}. Expanding $p_\sigma$ in the Hermite basis gives
\begin{align*}
\|p_\sigma\|_{\PP_\lambda}^2
&= \sum_{\alpha,\beta \in s^{-1}(\sigma)}
\hat p_\alpha \cdot \hat p_\beta \cdot \EE_{\PP_\lambda}[h_\alpha(Y)h_\beta(Y)] \\
&\le \sum_{\alpha,\beta \in s^{-1}(\sigma)}
\frac{\hat p_\alpha^2+\hat p_\beta^2}{2}\cdot 
\left|\EE_{\PP_\lambda}[h_\alpha(Y)h_\beta(Y)]\right|
\tag{$2ab\le a^2+b^2$} \\
&= \sum_{\alpha \in s^{-1}(\sigma)} \hat p_\alpha^2\cdot
\sum_{\beta \in s^{-1}(\sigma)}
\left|\EE_{\PP_\lambda}[h_\alpha(Y)h_\beta(Y)]\right|
\tag{symmetry of the double sum} \\
&\le d^{O(\sqrt{d})}\cdot\sum_{\alpha \in s^{-1}(\sigma)} \hat p_\alpha^2
\tag{\Cref{lem:row-sum-bound}} \\
&= d^{O(\sqrt{d})}\cdot \|p_\sigma\|_{\QQ}^2.
\tag{orthonormality under $\QQ$}
\end{align*}
This proves \eqref{eqn:norm-bound-sig}. Our proof of \Cref{lem:row-sum-bound} requires an auxiliary lemma bounding moments of the signal term in $\PP_\lambda$.
\begin{lemma}\label{lem:normalized-monomial-bound}
Let $\pi$ be a fixed subgaussian distribution with mean 0 and variance 1, and write
\[
m_q := \EE_{u \sim \pi}[|u|^q].
\]
Let $\gamma \in \NN^E$ with $|\gamma| \le d = n^{o(1)}$, and define $\deg_\gamma(i)$ to be the degree of vertex $i$ in the multigraph $\gamma$, with self-loops counting twice toward the degree. Write $v_1 :=|\deg_\gamma^{-1}(1)|$, the number of degree-1 vertices. We have
\begin{align*}
	|\EE[X^\gamma]|
&\le
O\!\left(n^{-|\gamma|/2-v_1}\right) \cdot (3|\gamma|\,m_3)^{v_1} \prod_{\deg_\gamma(i)>1} m_{\deg_\gamma(i)} \\
&\le d^{O(d)}\cdot n^{-|\gamma|/2 - v_1},
\end{align*}
where $X$ is distributed as the signal component~\eqref{eq:X} of $\PP_\lambda$.
\end{lemma}

In the first inequality, $O(\cdot)$ hides an absolute constant provided $n$ is sufficiently large, where the threshold for ``sufficiently large'' depends on $\pi$ and the growth rate of $|\gamma|$. The second inequality follows from the first using the bound on subgaussian moments (Definition~\ref{def:subg}).

The rest of the section is organized as follows. In \Cref{subsec:row-sum}, we prove \Cref{lem:row-sum-bound} using \Cref{lem:normalized-monomial-bound}. In \Cref{subsec:monomial-bound}, we prove \Cref{lem:normalized-monomial-bound}. Together these complete the proof of \Cref{lem:ratio-norms}.

\subsection{Proof of the row-sum bound: \Cref{lem:row-sum-bound}}\label{subsec:row-sum}

Fix $\sigma \in \Sigma_d$ and $\alpha \in s^{-1}(\sigma)$. We will bound the row sum
\[\sum_{\beta \in s^{-1}(\sigma)} \left| \EE_{\PP_\lambda}[h_\alpha(Y)h_\beta(Y)] \right|\]
by classifying $\beta$ into one of two cases. Consider the symmetric difference
\[
\alpha \triangle \beta := \alpha + \beta - 2(\alpha \wedge \beta),
\]
where $\alpha \wedge \beta$ denotes entrywise minimum. For an integer $0\leq k\leq |V(\alpha)|$, define the sets
\begin{align*}
\mathcal{T}_k
&:= \left\{
\beta \in s^{-1}(\sigma) :
|V(\beta)\setminus V(\alpha)| = k,\,
\core(\alpha \triangle \beta) = \emptyset,\,
|\alpha \triangle \beta| = 2k
\right\}, \\
\mathcal{R}_k
&:= \left\{
\beta \in s^{-1}(\sigma) :
|V(\beta)\setminus V(\alpha)| = k
\right\}\setminus \mathcal{T}_k.
\end{align*}
Intuitively, $\mathcal{T}_k$ consists of all $\beta$ that one can obtain by replacing any number of disjoint simple cycles from $\alpha$ with another collection of disjoint simple cycles (that is vertex-disjoint from $\alpha$). In aggregate across all $k$, $\mathcal{T}_k$ and $\mathcal{R}_k$ form a partition of the set of all $\beta\in s^{-1}(\sigma)$, so we can write
\[\sum_{\beta \in s^{-1}(\sigma)} \left| \EE_{\PP_\lambda}[h_\alpha(Y)h_\beta(Y)] \right| = \sum_{k=0}^{|V(\alpha)|}\left(\sum_{\beta\in \mathcal{T}_k} \left| \EE_{\PP_\lambda}[h_\alpha(Y)h_\beta(Y)] \right| + \sum_{\beta\in\mathcal{R}_k} \left| \EE_{\PP_\lambda}[h_\alpha(Y)h_\beta(Y)] \right|\right).\]

Before bounding the sum in each case, we will establish some identities that will be common to the argument in both cases. For any $\beta \in s^{-1}(\sigma)$, applying \Cref{prop:hermite-shift} and then using orthogonality of the Hermite basis under $Z$ gives the formula
\begin{align}
\EE_{\PP_\lambda}[h_\alpha(Y)h_\beta(Y)]
&= \EE_{X,Z}[h_\alpha(\lambda X+Z)h_\beta(\lambda X+Z)] \notag \\
&= \sum_{\kappa \le \alpha \wedge \beta}
c_{\alpha,\beta}(\kappa)\cdot \EE[(\lambda X)^{\alpha+\beta-2\kappa}],
\label{eq:ratio-hermite-gram}
\end{align}
where we set
\[
c_{\alpha,\beta}(\kappa)
:=
\frac{\kappa!}{\sqrt{\alpha!\beta!}}
\binom{\alpha}{\kappa}\binom{\beta}{\kappa}.
\]
For integers $0 \le k \le a \wedge b$, we have
\[
\frac{k!}{\sqrt{a!b!}}\binom{a}{k}\binom{b}{k}
= \frac{\sqrt{a!b!}}{k!(a-k)!(b-k)!}
\le \sqrt{\binom{a}{k}\binom{b}{k}}
\le 2^{(a+b)/2}.
\]
Taking a product over coordinates, and using $|\alpha|=|\beta|\le d$, we obtain
\begin{equation}\label{eq:ratio-coeff-bound}
c_{\alpha,\beta}(\kappa) \le 2^d.
\end{equation}
Similarly,
\begin{equation}\label{eq:ratio-kappa-count}
|\{\kappa\in \NN^E : \kappa \le \alpha \wedge \beta\}|
= \prod_{e \in E}(\alpha_e \wedge \beta_e + 1)
\le \prod_{e \in E} 2^{\alpha_e \wedge \beta_e}
\le 2^d.
\end{equation}

\paragraph{Case 1: bounding the sum over $\mathcal{R}_k$.}

\begin{lemma}\label{lem:ratio-remainder-class}
Consider the setting of \Cref{lem:ratio-norms}. Uniformly over all $0 \le k \le |V(\alpha)|$, we have
\[
\sum_{\beta \in \mathcal{R}_k}
\left|\EE_{\PP_\lambda}[h_\alpha(Y)h_\beta(Y)]\right|
= o(1).
\]
\end{lemma}
\begin{proof}
	We begin with a crude bound on $|\mathcal{R}_k|$. One can specify a member $\beta\in \mathcal{R}_k$ with the following two pieces of information:
	\begin{enumerate}
		\item The $k$ new vertices in $V(\beta)\setminus V(\alpha)$. There are at most $n^k$ ways to pick them.
		\item The edges of $\beta$. Given $V(\beta)\setminus V(\alpha)$, every edge of $\beta$ has both endpoints in \(V(\beta)\subseteq V(\alpha)\cup (V(\beta)\setminus V(\alpha))\), a known set of \(O(d)\) vertices. Since $|\beta|\le d$, there are at most $d^{O(d)}$ possible choices for these edges, including multiplicities.
	\end{enumerate}
	Therefore we have
\begin{equation}\label{eq:ratio-remainder-count}
|\mathcal{R}_k| \le n^k d^{O(d)}.
\end{equation}
Now fix $\beta \in \mathcal{R}_k$ and $\kappa \le \alpha \wedge \beta$ (as in \eqref{eq:ratio-hermite-gram}). Set
\[
\gamma := \alpha + \beta - 2\kappa.
\]
Let
\[
v_1 := |\deg_\gamma^{-1}(1)|,\qquad
v_2 := |\deg_\gamma^{-1}(2)|,\qquad
v_{\ge 3} := |\{i : \deg_\gamma(i)\ge 3\}|.
\]
Summing degrees yields
\begin{align*}
|\gamma| + 2v_1
&= \sum_i \deg_\gamma(i)/2 + 2v_1
\tag{$\sum_i \deg_\gamma(i)=2|\gamma|$} \\
&\ge \frac{3}{2}v_1 + v_2 + \frac{3}{2}v_{\ge 3}\\
&= |V(\gamma)| + \frac{v_1+v_{\ge 3}}{2}\\
&\geq |V(\alpha\triangle\beta)|+ \frac{v_1+v_{\ge 3}}{2}.\tag{$\alpha\triangle\beta\leq \gamma$ because $\kappa\leq \alpha\wedge\beta$}\\
\end{align*}
	We claim that \(|\gamma|+2v_1\) is at least $2k+1$. We will prove this by considering three possible cases.
\begin{enumerate}
	\item If $|V(\alpha\triangle\beta)|\geq 2k+1$, then we are immediately done by the above.
	\item If $|\alpha\triangle\beta|\geq 2k+1$, then we are also done because $|\gamma|\geq |\alpha\triangle\beta|\geq 2k+1$.
	\item Otherwise, we have $|V(\alpha\triangle\beta)|\leq 2k$ and $|\alpha\triangle\beta|\leq 2k$. In fact, since $V(\alpha\triangle\beta)\supseteq V(\alpha)\triangle V(\beta)$, we must have
		\[|V(\alpha\triangle\beta)|\geq |V(\alpha)\triangle V(\beta)| = |V(\alpha)\setminus V(\beta)| + |V(\beta)\setminus V(\alpha)| = 2k,\]
		where we used that $|V(\alpha)\setminus V(\beta)| = |V(\beta)\setminus V(\alpha)|$ because $|V(\alpha)| = |V(\beta)|$. So it must be the case that $|V(\alpha\triangle\beta)| = 2k$ and $|\alpha\triangle\beta|\leq 2k$.
		
			Given this, we can conclude that \(\alpha\triangle\beta\) is not a vertex-disjoint union of simple cycles; in particular if $|\alpha\triangle\beta|< 2k$, the average degree is strictly smaller than $2$, and if $|\alpha\triangle\beta| = 2k$, the definition of $\mathcal{R}_k$ implies that $\core(\alpha\triangle\beta)$ is nonempty, and the claim follows.
		
			In particular, \(\alpha\triangle\beta\) has at least one vertex with degree not in \(\{0,2\}\). Since \(\gamma-\alpha\triangle\beta=2(\alpha\wedge\beta-\kappa)\) changes every vertex degree by an even nonnegative amount, this vertex still has degree not in \(\{0,2\}\) in \(\gamma\). Thus \(v_1+v_{\ge3}\ge1\), and the above inequalities imply
			\[
			|\gamma|+2v_1\ge |V(\alpha\triangle\beta)|+\frac{v_1+v_{\ge3}}{2}>2k.
			\]
			Since the left-hand side is an integer, it is at least \(2k+1\).
\end{enumerate}
With this bound, we can apply \Cref{lem:normalized-monomial-bound} to obtain
\[
\left|\EE[X^\gamma]
\right|
\le d^{O(d)} n^{-k-1/2}.
\]
Now we use \eqref{eq:ratio-hermite-gram}, \eqref{eq:ratio-coeff-bound}, \eqref{eq:ratio-kappa-count} and the assumption $\lambda \le 1$ to arrive at
\[
\left|\EE_{\PP_\lambda}[h_\alpha(Y)h_\beta(Y)]\right|
\le d^{O(d)} n^{-k-1/2}.
\]
Combining this with \eqref{eq:ratio-remainder-count},
\[
\sum_{\beta \in \mathcal{R}_k}
\left|\EE_{\PP_\lambda}[h_\alpha(Y)h_\beta(Y)]\right|
\le \left(n^k d^{O(d)}\right)\cdot d^{O(d)} n^{-k-1/2}
= d^{O(d)} n^{-1/2} = o(1),
\]
where we used $d = o(\log n / \log\log n)$.
\end{proof}

\paragraph{Case 2: bounding the sum over $\mathcal{T}_k$.}

\begin{lemma}\label{lem:ratio-tight-class}
Consider the setting of \Cref{lem:ratio-norms}. Uniformly over all $0 \le k \le |V(\alpha)|$, we have
\[
\sum_{\beta \in \mathcal{T}_k}
\left|\EE_{\PP_\lambda}[h_\alpha(Y)h_\beta(Y)]\right|
\le d^{O(\sqrt{d})}.
\]
\end{lemma}
\begin{proof}
As in the previous case, we will first prove a bound on $|\mathcal{T}_k|$, and then prove a bound on each term in the sum. By definition of $\mathcal{T}_k$, for any $\beta\in \mathcal{T}_k$, the graphs 
\[
A=\alpha-\alpha\wedge\beta
\qquad\text{and}\qquad
B=\beta-\alpha\wedge\beta
\]
must each be a vertex-disjoint union of cycles such that \(V(A)\), \(V(B)\), and \(V(\alpha\wedge\beta)\) are pairwise disjoint. Indeed, \(A+B=\alpha\triangle\beta\) is a vertex-disjoint union of simple cycles with \(2k\) edges, hence has \(2k\) vertices; since it contains \(V(\alpha)\triangle V(\beta)\), which also has size \(2k\), it cannot contain any vertex of \(V(\alpha\wedge\beta)\). Furthermore the fact that $C^\alpha = C^\beta$ implies that $C^{A} = C^{B}$. Also, note that $\beta$ is completely determined by $A$ and $B$. With these basic facts, let us describe one way to specify $\beta\in \mathcal{T}_k$ in four steps.
\begin{enumerate}
	\item Pick a function \(C':\NN_{\geq 1}\to \NN\) such that $C' \le C^\alpha$ with $\sum_{r \ge 1} r C'(r)=k$.
	\item Pick the edges in $A$. To do this, for each $r$, we must choose \(C'(r)\) of the $C^\alpha(r)$ simple $r$-cycles from $\alpha$. The number of ways to do this is
	\[
	\prod_{r \in [d]} \binom{C^\alpha(r)}{C'(r)}
	\le \prod_{r \in [d]}
	\frac{C^\alpha(r)^{C'(r)}}{C'(r)!}.
	\]
	\item Pick the $k$ vertices in $V(\beta)\setminus V(\alpha)$. The number of ways to do this is at most $\binom{n}{k}\leq n^k/k!$.
	\item Pick the edges in $B$. This can be done by (1) partitioning the $k$ chosen vertices into a disjoint collection of subsets so that exactly $C'(r)$ subsets have size $r$ for each $r\in [d]$, and then (2) specifying a cycle on each of these subsets. The first sub-step can be done in
		\[\frac{k!}{\prod_{r\in [d]}{(r!)^{C'(r)}\cdot C'(r)!}}\]
		ways, and the second sub-step can be done in
		\[\prod_{r\geq 3}\left(\frac{(r-1)!}{2}\right)^{C'(r)}\]
		ways. Multiplying these together and bounding crudely, this number of ways to perform this step is at most
		\[\frac{k!}{\prod_{r\in [d]}C'(r)!}.\]
\end{enumerate}
Therefore
\begin{align*}
|\mathcal{T}_k|
&\le \sum_{C' \le C^\alpha}
\frac{n^k}{k!} \cdot k!\cdot \prod_{r \in [d]}
\frac{C^\alpha(r)^{C'(r)}}{C'(r)!^2} \\
&= n^k \cdot\prod_{r \in [d]}
\sum_{c \le C^\alpha(r)} \frac{C^\alpha(r)^c}{c!^2}
\tag{swapped $\sum$ and $\prod$} \\
&\le n^k \cdot\prod_{r \in [d]} \exp(2\sqrt{C^\alpha(r)})
\tag{$\sum_{c\ge 0} x^c/c!^2 \le e^{2\sqrt{x}}$} \\
&= n^k\cdot \exp\!\left(2\sum_{r \in [d]} \sqrt{C^\alpha(r)}\right)\\
&\le n^k \cdot\exp\!\left(
2\sqrt{\sum_{r \in [d]} r \cdot C^\alpha(r)}
\sqrt{\sum_{r \in [d]} r^{-1}}
\right)
\tag{Cauchy--Schwarz} \\
&\le n^k \cdot\exp(O(\sqrt{d\log d}))
\tag{$\sum_r r\cdot C^\alpha(r)\le d$ and Harmonic number asymptotics} \\
&\le n^k \cdot d^{O(\sqrt{d})}.
\end{align*}
Having proved a bound on $|\mathcal{T}_k|$, we now bound each term in the row sum. By \eqref{eq:ratio-hermite-gram} we can write
\begin{align*}
\EE_{\PP_\lambda}[h_\alpha(Y)h_\beta(Y)]
&= \sum_{\kappa \le \alpha \wedge \beta}
c_{\alpha,\beta}(\kappa)\cdot \EE[(\lambda X)^{\alpha+\beta-2\kappa}] \\
&= c_{\alpha,\beta}(\alpha\wedge\beta)\cdot \EE[(\lambda X)^{\alpha\triangle\beta}] + \sum_{\kappa \lneq \alpha \wedge \beta}
c_{\alpha,\beta}(\kappa)\cdot \EE[(\lambda X)^{\alpha+\beta-2\kappa}],
\end{align*}
where $\lneq$ means that $\le$ holds entrywise and the inequality is strict on at least one coordinate.

Using the first inequality in \Cref{lem:normalized-monomial-bound}, we can bound $|\EE[X^{\alpha\triangle\beta}]|$ by \(O(1)\cdot n^{-k}\) because all vertices in $\alpha\triangle\beta$ have degree in $\{0,2\}$ (and note $m_2 = 1$). For \(\kappa \lneq \alpha\wedge\beta\), write \(\eta := \alpha\wedge\beta-\kappa\). Then \(\alpha+\beta-2\kappa = \alpha\triangle\beta+2\eta\), where \(\eta\neq 0\). Thus there are no degree-\(1\) vertices and \(|\alpha+\beta-2\kappa|/2 \ge k+1\), so \Cref{lem:normalized-monomial-bound} gives
\[
{|\EE[X^{\alpha+\beta-2\kappa}]|
\leq d^{O(d)}\cdot n^{-k-1}.}
\]

Let us also calculate
\begin{align*}
	c_{\alpha,\beta}(\alpha\wedge\beta) &= \prod_{e\in E}\frac{(\alpha\wedge\beta)_e!}{\sqrt{\alpha_e!\beta_e!}}\binom{\alpha_e}{(\alpha\wedge\beta)_e}\binom{\beta_e}{(\alpha\wedge\beta)_e}.
\end{align*}
For all $e$ satisfying $(\alpha\triangle\beta)_e= 0$, the contribution to the product above is exactly $1$. Otherwise, since $\beta\in \mathcal{T}_k$, $\{\alpha_e,\beta_e\}$ is forced to be either $\{0,1\}$ {if \(e\) appears once in a simple cycle} or $\{0,2\}$ {if \(e\) is the edge of a simple \(2\)-cycle}. Therefore the contribution is either $1$ or $1/\sqrt{2}$, and in particular at most $1$. So we have $c_{\alpha,\beta}(\alpha\wedge\beta)\leq 1$. Plugging in these bounds and recalling $\lambda \le 1$, we have
\begin{align*}
	\lvert\EE_{\PP_\lambda}[h_\alpha(Y)h_\beta(Y)]\rvert
		&\leq c_{\alpha,\beta}(\alpha\wedge\beta)\cdot {\lvert}\EE[X^{\alpha\triangle\beta}]{\rvert} + \sum_{\kappa \lneq \alpha \wedge \beta}
c_{\alpha,\beta}(\kappa)\cdot {\lvert}\EE[X^{\alpha+\beta-2\kappa}]{\rvert}\\
	&\leq O(1)\cdot n^{-k} + d^{O(d)}\cdot n^{-k-1}\cdot \sum_{\kappa \lneq \alpha \wedge \beta}
c_{\alpha,\beta}(\kappa)\\
	&\leq O(1)\cdot n^{-k} + d^{O(d)}\cdot n^{-k-1}\tag{by \eqref{eq:ratio-coeff-bound}, \eqref{eq:ratio-kappa-count}}\\
	&\leq O(1)\cdot n^{-k}.\tag{$d = o(\log n / \log\log n)$}
\end{align*}
		Recalling $|\mathcal{T}_k| \le n^k \cdot d^{O(\sqrt{d})}$, the sum over all $\beta\in \mathcal{T}_k$ is bounded by \(O(1)\cdot n^{-k}\cdot |\mathcal{T}_k| \leq d^{O(\sqrt{d})}\).
\end{proof}

\paragraph{Completing the proof.} To complete the proof of \Cref{lem:row-sum-bound}, we sum the inequalities in \Cref{lem:ratio-remainder-class,lem:ratio-tight-class} over all $0 \le k \le {2}d$ to obtain
\begin{align*}
\sum_{\beta \in s^{-1}(\sigma)}
\left|\EE_{\PP_\lambda}[h_\alpha(Y)h_\beta(Y)]\right|
&\le \sum_{k=0}^{{2}d}
\left(
d^{O(\sqrt{d})} + o(1)
\right)  = d^{O(\sqrt{d})}.\\
\end{align*}

\subsection{Proof of the moment bound: \Cref{lem:normalized-monomial-bound}}\label{subsec:monomial-bound}

Write $r := |\gamma|$ and assume $r \ge 1$ because the case $r=0$ is immediate. Let
\[
S := V(\gamma) = \{i : \deg_\gamma(i)>0\}.
\]
For $i \in S$, write $d_i := \deg_\gamma(i)$. Since $\sum_{i \in S} d_i = 2r$, we have $|S| \le 2r = n^{o(1)}$.

Recalling the definition of $X$ from \eqref{eq:X}, we can write
\[
X^\gamma
= c_\gamma \cdot \mathbf{1}\{x \neq 0\}\cdot
	\left(\frac{\sqrt{n}}{\|x\|_2^2}\right)^r\cdot
\prod_{i \in S} x_i^{d_i},
\]
where $|c_\gamma| \le 1$ absorbs the diagonal $1/\sqrt{2}$ scaling. Define the random variable
\[
T := \sum_{j \notin S} x_j^2.
\]
\(T\) is a sum of \(n-|S|=n(1-o(1))\) independent squares \(x_j^2\), where each \(x_j\) is drawn from \(\pi\). Applying \Cref{prop:subgaussian-square-lower-tail} with $m=n-|S|$ and $t=n^{3/4}$, we obtain
\[
\Pr\!\left[T < (n-|S|)-n^{3/4}\right]
\le \exp(-c n^{1/2})
\]
for some constant $c=c(\pi)>0$ and all sufficiently large $n$. Define the ``good'' event
\[
G := \{T \ge (n-|S|)-n^{3/4}\}.
\]
We first bound the contribution from $G^c$. By weighted AM--GM,
\[
\prod_{i \in S} |x_i|^{d_i}
= \left(\prod_{i \in S}(x_i^2)^{d_i/(2r)}\right)^r
\le \left(\sum_{i \in S}\frac{d_i}{2r}x_i^2\right)^r
\le \|x\|_2^{2r}.
\]
Consequently,
\[
|X^\gamma|
\le |c_\gamma|
\left(\frac{\sqrt{n}}{\|x\|_2^2}\right)^r
\|x\|_2^{2r}
\le n^{r/2},
\]
and therefore
\[
|\EE[X^\gamma \mathbf{1}_{G^c}]|
\le n^{r/2}\exp(-c n^{1/2}).
\]
This satisfies the bound claimed in Lemma~\ref{lem:normalized-monomial-bound}, using the following observations. First, $v_1 \le 2r$, which means $v_1$ and $r = |\gamma|$ are both $n^{o(1)}$. Second, \(m_q\ge 1\) for every \(q\ge2\) due to Jensen combined with \(\EE_{u \sim \pi}[u^2]=1\).

We now turn to the contribution from $G$. For $t>0$, define
\[
B(t)
:= \Eop_{x_S}\left[
\left(t+\sum_{i \in S} x_i^2\right)^{-r}
\prod_{i \in S} x_i^{d_i}
\right].
\]
Since $T$ is independent of $x_S$,
\[
\EE[X^\gamma \mathbf{1}_G]
= c_\gamma n^{r/2} \cdot \Eop_T[\mathbf{1}_G \cdot B(T)].
\]
Using the identity
\[
u^{-r}
= \frac{1}{(r-1)!}\int_0^\infty s^{r-1}e^{-su}\,ds
\qquad (u>0),
\]
we obtain
\begin{align*}
B(t)
&= \frac{1}{(r-1)!}
\int_0^\infty s^{r-1}e^{-st}
\Eop_{x_S}\left[\prod_{i \in S} x_i^{d_i}e^{-s x_i^2}\right]\,ds \\
&= \frac{1}{(r-1)!}
\int_0^\infty s^{r-1}e^{-st}
\prod_{i \in S}\zeta(s,d_i)\,ds,
\end{align*}
where
\[
\zeta(s,q) := \EE_{u \sim \pi}[u^q e^{-s u^2}].
\]
For any $s \ge 0$, we claim that
\[
|\zeta(s,1)| \le m_3\cdot s,
\qquad
|\zeta(s,q)| \le m_q \quad\text{for every integer } q \ge 2.
\]
For the first claim, using $\EE_{u \sim \pi}[u]=0$,
\begin{align*}
|\zeta(s,1)|
&= \left|\EE_{u \sim \pi}[u(e^{-s u^2}-1)]\right|
\tag{$\EE[u]=0$} \\
&\le \EE_{u \sim \pi}[|u| \cdot |e^{-s u^2}-1|]
\tag{triangle inequality} \\
&\le s \EE_{u \sim \pi}[|u|^3]
\tag{$1-e^{-x}\le x$ for $x\ge 0$}
= m_3\cdot s.
\end{align*}
Now for the second claim,
\[ |\zeta(s,q)| \le \EE_{u \sim \pi} |u^q e^{-su^2}| \le \EE_{u \sim \pi} |u^q| = m_q. \]
Therefore
\begin{align*}
|B(t)|
&\le \frac{1}{(r-1)!}\cdot
m_3^{v_1}\cdot
\prod_{\deg_\gamma(i)>1} m_{\deg_\gamma(i)}\cdot 
\int_0^\infty s^{r-1+v_1}e^{-st}\,ds \\
&=
\frac{(r+v_1-1)!}{(r-1)!}\cdot
m_3^{v_1}\cdot
\prod_{\deg_\gamma(i)>1} m_{\deg_\gamma(i)}
\cdot t^{-r-v_1}.
\end{align*}
Since $v_1 \le 2r$,
\[
\frac{(r+v_1-1)!}{(r-1)!}
= \prod_{j=0}^{v_1-1}(r+j)
\le (3r)^{v_1},
\]
and hence
\[
|B(t)|
\le (3rm_3)^{v_1}\cdot 
\prod_{\deg_\gamma(i)>1} m_{\deg_\gamma(i)}
\cdot t^{-r-v_1}.
\]
On $G$ we have $T \ge n(1-O(n^{-1/4}))$, and recall $r+v_1=n^{o(1)}$, so $T^{-r-v_1} = O(n^{-r-v_1})$. Therefore
\[
|\EE[X^\gamma \mathbf{1}_G]|
\le O\!\left(n^{-r/2-v_1}\right)
\prod_{\deg_\gamma(i)>1} m_{\deg_\gamma(i)}
\cdot (3rm_3)^{v_1}.
\]
Combining this with the bound on $G^c$ proves the lemma.

\section{Large-Set Expansion for Spiked Wigner}
\label{app:lse}

This section is devoted to the proof of Lemma~\ref{lem:lse}.

\subsection{Proof overview}
\label{sec:lse-overview}

For a deterministic choice of $x\in\RR^n$, let $\PP_{\lambda,x}$ denote the spiked Wigner distribution (Definition~\ref{def:wigner-additive}) with the spike fixed to $x$ instead of chosen at random. That is, \(\PP_{\lambda,x}\) is distributed as \(Y = \mu_\lambda(x) + Z\) where \(Z\sim \cN(0, I_N)\) and
\[
\mu_\lambda(x)
:= \mathbf{1}\{x\neq 0\}\cdot \lambda \sqrt{n}\cdot \frac{\Phi(xx^\top)}{\|x\|^2}.
\]
The proof has two parts. First, in \Cref{sec:fixed-signal-lse} we establish expansion for $\PP_{\lambda,x}$ by appealing to known results on Gaussian space. Then, in \Cref{sec:mixture} we use this to show large-set expansion for the mixture $\PP_\lambda=\EE_{x\sim \pi^{\otimes n}}[\PP_{\lambda,x}]$, with a core technical claim deferred to \Cref{sec:transfer}. This last step uses permutation invariance and requires some delicate second-moment arguments to show that ``nearby'' signals induce spiked distributions that are statistically ``close enough'' (see Remark~\ref{rem:lr-choice}).

\subsection{Deterministic signal}
\label{sec:fixed-signal-lse}

\begin{lemma}[Expansion for deterministic signal]
\label{lem:gaussian-expansion}
For every \(\rho \in (0,1)\) and every \(a>0\), there exists a constant \(c>0\), depending only on \(\rho\) and \(a\), such that the following holds. Fix \(0<\lambda<\lambda'\) with \(\lambda/\lambda'=\rho\). For every \(x\in \RR^n\) and every \(\beta\ge 0\), if \(g:\RR^N \to [0,1]\) satisfies
\[
\EE_{\PP_{\lambda, x}}[g] \ge \beta,
\]
then
\[
\PP_{\lambda', x}\!\left(T_\rho g \ge \beta^{c}\right)\ge 1-\beta^{a}.
\]
\end{lemma}

We note that while this bound is good enough for our purposes, it is not sharp: for instance, the conclusion is vacuous when $\beta = 1$. See Remark~\ref{rem:iso} for further discussion.

\begin{proof}
If \(\beta=0\), the claim is immediate, so assume \(\beta>0\).
The hypothesis also forces \(\beta\le 1\).
Note that \(\mu_\lambda(x)=\rho\mu_{\lambda'}(x)\). Define
\[
\widetilde g(z):=g(\mu_\lambda(x)+z).
\]
Then, under $\QQ = \cN(0,I_N)$ we have \(\EE_{\QQ}[\widetilde g]=\EE_{\PP_{\lambda,x}}[g]\ge \beta\) and
\[
T_\rho\widetilde g(z)=T_\rho g(\mu_{\lambda'}(x)+z).
\]
It remains to prove the corresponding Gaussian statement for \(\widetilde g\). The two-function Gaussian reverse hypercontractivity inequality (e.g.~\cite[Exercise~10.6]{odonnell-book} or~\cite[Theorem~1,~Example~1]{CDP-gaussian-holder})
states that for every $\rho \in (0,1)$ and all nonnegative \(f,h:\RR^N\to\RR\),
\begin{equation}\label{eq:rev-hyp}
\langle f,T_\rho h\rangle
\geq
\|f\|_{1-\rho}\|h\|_{1-\rho}.
\end{equation}
Here we are using the inner product $\langle \cdot,\cdot \rangle = \langle \cdot,\cdot \rangle_\QQ$ and $p$-norm with respect to $\QQ$: $\|f\|_p := \EE_\QQ[f^{p}]^{1/p}$. We apply~\eqref{eq:rev-hyp} with
\[
f=\mathbf{1}\{T_\rho \widetilde g<\beta^{c}\},
\qquad
h=\widetilde g
\]
for a constant \(c\) to be specified later. Write \(q:=\QQ(T_\rho \widetilde g<\beta^{c})\). Then
\[
\langle f,T_\rho h\rangle
=
\EE_{\QQ}[f\,T_\rho \widetilde g]
\le
q\,\beta^{c},
\]
since \(f=1\) only on the event \(\{T_\rho \widetilde g<\beta^{c}\}\). On the other hand,
\begin{equation}\label{eq:apply-rh}
\langle f, T_\rho h\rangle
\ge
\|f\|_{1-\rho}\|\widetilde g\|_{1-\rho}
\ge
(\EE_\QQ[f])^{1/(1-\rho)} (\EE_\QQ[\widetilde g])^{1/(1-\rho)}
\ge
q^{1/(1-\rho)}\beta^{1/(1-\rho)},
\end{equation}
since \(f\) and \(\widetilde g\) have outputs in \([0,1]\). Combining the two inequalities above and setting
\[
c=\frac{1+\rho a}{1-\rho}
\]
yields \(q\le \beta^a\). Thus
\[
\PP_{\lambda',x}(T_\rho g\ge \beta^c)
=
\QQ(T_\rho\widetilde g\ge \beta^c)
\ge
1-\beta^a. \qedhere
\]
\end{proof}

\begin{remark}\label{rem:iso}
A key inequality in~\eqref{eq:apply-rh} above was a bound on $\langle f,T_\rho h\rangle$ in terms of $\EE[f]$ and $\EE[h]$, where $f,g$ are $[0,1]$-valued functions with standard Gaussian inputs. While not needed for our purposes, we note that sharper --- in fact, \emph{exact} --- bounds of this form are known, as a consequence of the \emph{two-function Borell isoperimetric theorem}~\cite{borell,MN-robust} (see~\cite{odonnell-book}, pg.~375 combined with Exercise~11.19c). Specifically, for fixed values of $\EE[f]$ and $\EE[h]$, the quantity $\langle f,T_\rho h\rangle$ is minimized by taking $f,h$ to be indicators for opposing half-spaces, and maximized by taking $f,h$ to be indicators for nested half-spaces.
\end{remark}

\subsection{Extending to a mixture of signals}
\label{sec:mixture}

We now turn to the proof of Lemma~\ref{lem:lse}, now considering a mixture of signals rather than a deterministic one. Recall that $\pi$ is supported on a finite set, which we denote by $\cS = \{s_1,\ldots,s_m\}$ where $m := |\cS|$.

For \(B\ge 1\), call a distribution \(\nu\) over \(\RR\) \(B\)-\emph{typical} if it is supported on \(\cS\) and
\[
|\nu(s_j)-\pi(s_j)|\le \frac{B}{\sqrt n}
\qquad\text{for every }j\in[m].
\]

\begin{claim}
\label{claim:finite-support-typicality}
If \(x\sim \pi^{\otimes n}\) and \(\nu_x\) denotes its empirical distribution, then
\[
\Pr(\nu_x\ \mathrm{is}\ B\text{-}\mathrm{typical}) \ge 1-2m\cdot \exp(-2B^2).
\]
\end{claim}

\begin{proof}
Fix \(j\in[m]\). Then \(n\nu_x(s_j)\sim \mathrm{Binomial}(n,\pi(s_j))\), so Hoeffding's inequality gives
\[
\Pr\!\left(\left|\nu_x(s_j)-\pi(s_j)\right|>\frac{B}{\sqrt n}\right)
\le 2\exp(-2B^2).
\]
A union bound over \(j\in[m]\) proves the claim.
\end{proof}

The following is a key claim that will allow us to compare the spiked models induced by two different ``typical'' signals.

\begin{claim}
\label{claim:transfer}
For any $\lambda \ge 0$ there are constants \(b,C>0\), depending only on $\lambda$ and $\pi$, such that if \(B\le b\sqrt n\) and \(\nu_x,\nu_{\hat x}\) are two \(B\)-typical empirical distributions for some \(x,\hat x\in \RR^n\), the following holds. For any event $A \subseteq \RR^N$ that is permutation-invariant in the sense that $y \in A \implies \sigma \cdot y \in A$ for all $\sigma \in S_n$ (see Definition~\ref{def:perm-inv}),
\[
{\PP}_{\lambda,x}(A)\le \exp(CB^2)\,{\PP}_{\lambda,\hat x}(A)^{1/2}.
\]
\end{claim}
We defer the proof of \Cref{claim:transfer} to \Cref{sec:transfer}. We will first show how it implies the desired result.

\begin{proof}[Proof of \cref{lem:lse}]
Let \(b,C\) be the constants from \Cref{claim:transfer} applied at \(\lambda'\), and set \(a:=2C+4\).
Let \(c'>0\) be the constant supplied by \Cref{lem:gaussian-expansion} with failure exponent \(a\), and set \(c:=2c'\). Define
\[
\eps_0
:=
\min\left\{
\frac12,
\frac{1}{2(4m)^C}
\right\}.
\]
For a given \(\eps\in(0,\eps_0]\), set
\[
B:=\sqrt{\log(4m/\eps)}.
\]
The definition of \(\eps_0\) ensures \(B\ge 1\) and \((4m)^C\eps\le 1/2\).
Let \(\mathfrak T_B\) denote the set of \(B\)-typical empirical distributions on \(\cS\). Choose $\kappa \in (0,b^2/\log 2)$ so that (for sufficiently large $n$) the lower bound \(\eps\ge 2^{-\kappa n}\) ensures that \(B\le b\sqrt n\), as required by \Cref{claim:transfer}. Writing \(\PP_\lambda=\EE_{x\sim \pi^{\otimes n}}[\PP_{\lambda,x}]\), we have
\[
\eps\le \EE_{\PP_\lambda}[g]
=
\EE_x\!\left[\EE_{\PP_{\lambda,x}}[g]\cdot \mathbf{1}\{\nu_x\in \mathfrak T_B\}\right]
\;+\;
\EE_x\!\left[\EE_{\PP_{\lambda,x}}[g]\cdot \mathbf{1}\{\nu_x\notin \mathfrak T_B\}\right].
\]
Applying \cref{claim:finite-support-typicality} yields
\[
\eps
\le
\EE_x\!\left[\EE_{\PP_{\lambda,x}}[g]\cdot \mathbf{1}\{\nu_x\in \mathfrak T_B\}\right]
\, + \, 2m\exp(-2B^2),
\]
since \(0\le g\le 1\). With our choice of \(B\), the error term is
\[
2m\exp(-2B^2)
=
\frac{\eps^2}{8m}
\le \eps/2.
\]
Thus
\[
\EE_x\!\left[\EE_{\PP_{\lambda,x}}[g]\cdot \mathbf{1}\{\nu_x\in \mathfrak T_B\}\right]\ge \frac{\eps}{2},
\]
meaning there exists \(x_*\in\RR^n\) such that \(\nu_{x_*}\in \mathfrak T_B\) and
\[
\EE_{\PP_{\lambda,x_*}}[g]\ge \frac{\eps}{2}.
\]
Applying \cref{lem:gaussian-expansion} gives
\[
\PP_{\lambda',x_*}(T_\rho g< (\eps/2)^{c'})\le (\eps/2)^a.
\]
Since \(g\) is permutation-invariant, the event \(\{T_\rho g<(\eps/2)^{c'}\}\) is also permutation-invariant. Using \Cref{claim:finite-support-typicality} for the atypical profiles and \Cref{claim:transfer} with reference profile \(x_*\) for the typical profiles,
\begin{align*}
	\PP_{\lambda'}(T_\rho g < (\eps/2)^{c'}) &= \EE_{x\sim \pi^{\otimes n}}[{\PP}_{\lambda',x}(T_\rho g < (\eps/2)^{c'})]\\
	&\leq \Pr_{x\sim \pi^{\otimes n}}[\nu_x\notin \mathfrak T_B] + \max_{x \, : \, \nu_x\in \mathfrak T_B}{\PP}_{\lambda',x}(T_\rho g < (\eps/2)^{c'})\\
	&\leq 2m\cdot \exp(-2B^2) + \exp(CB^2)\cdot (\eps/2)^{a/2}\\
	&\le \eps/2 + (4m)^C\eps^{a/2-C}.
\end{align*}
Since \(a/2-C=2\), the last display is at most \(\eps/2+(4m)^C\eps^2\le \eps\).
Also, since \(\eps_0\le 1/2\) and \(c=2c'\), we have \(\eps^c\le(\eps/2)^{c'}\) for every \(\eps\le \eps_0\). Therefore
\[
\PP_{\lambda'}(T_\rho g < \eps^{c})\le \eps,
\]
as claimed.
\end{proof}

\subsection{Proof of the transfer claim: \Cref{claim:transfer}}
\label{sec:transfer}

\begin{proposition}
\label{prop:finite-support-norm}
If \(z\in\RR^n\) has \(B\)-typical empirical distribution, then
\[
\bigl|\|z\|_2^2-n\bigr|\le \sum_{j=1}^m s_j^2\cdot B\sqrt{n}.
\]
\end{proposition}

\begin{proof}
Write \(\nu_z\) for the empirical distribution of \(z\). Since \(\sum_j \pi(s_j)s_j^2=1\),
\[
\|z\|_2^2-n
=
n\sum_{j=1}^m (\nu_z(s_j)-\pi(s_j))s_j^2.
\]
Therefore,
\[
\bigl|\|z\|_2^2-n\bigr|
\le
n\sum_{j=1}^m \frac{B}{\sqrt n}s_j^2
=
\sum_{j=1}^m s_j^2\cdot B\sqrt{n}.
\]
\end{proof}

\begin{proof}[Proof of \Cref{claim:transfer}]
Fix \(\lambda\ge 0\). Let \(b>0\) be small enough that
\[
b\le
\min\left\{
\frac12\min_{i\in[m]}\pi(s_i), \,
\frac{1}{2\sum_{i=1}^m s_i^2}
\right\}.
\]
Then every \(B\)-typical empirical distribution with \(B\le b\sqrt n\) assigns mass at least \(\pi(s_i)/2\) to each \(s_i\).

Fix \(x,\hat x\in \RR^n\) with $B$-typical empirical distributions as in the claim. We will describe a procedure for sampling a random vector $y$ that is ``close'' to $\hat{x}$ but has the empirical distribution $\nu_x$.
For each \(i\in[m]\), let
\[
I_i:=\{k\in[n]:\hat x_k=s_i\},
\qquad
\Delta_i:=n\max\{0,\nu_{\hat x}(s_i)-\nu_x(s_i)\}.
\]
Let \(J:=\{i\in[m]:\Delta_i>0\}\). For each \(i\in J\), choose a uniformly random subset \(R_i\subseteq I_i\) of size \(\Delta_i\), independently across \(i\), and write
\[
R:=\bigcup_{i\in J} R_i.
\]
Thus \(R\) records the coordinates where \(\hat x\) has too many copies of a symbol relative to \(\nu_x\). Let \(T\) be the multiset containing \(n\max\{0,\nu_x(s_i)-\nu_{\hat x}(s_i)\}\) copies of \(s_i\) for each \(i\). Since \(\nu_x\) and \(\nu_{\hat x}\) are both probability measures, $|R| = |T|$. Form a random vector \(y\) as follows: choose $R$ randomly as above, delete the entries of \(\hat x\) indexed by \(R\), and then refill those \(|R|\) vacant positions with the values from $T$, inserted in a uniformly random order. For each \(i\in[m]\), this removes exactly \(n\max\{0,\nu_{\hat x}(s_i)-\nu_x(s_i)\}\) copies of \(s_i\) from \(\hat x\) and then inserts exactly \(n\max\{0,\nu_x(s_i)-\nu_{\hat x}(s_i)\}\) copies of \(s_i\). Hence the total number of \(s_i\)-valued coordinates in \(y\) is
\[
n\nu_{\hat x}(s_i)
-n\max\{0,\nu_{\hat x}(s_i)-\nu_x(s_i)\}
+n\max\{0,\nu_x(s_i)-\nu_{\hat x}(s_i)\}
=
n\nu_x(s_i),
\]
so \(y\) has empirical distribution \(\nu_x\) with probability \(1\). Let \(\PP_{\lambda,\hat x\to \nu_x} := \EE_{y}[\PP_{\lambda, y}]\) denote the corresponding planted law.

Our aim will be to show that the distributions $\PP_{\lambda,\hat x\to \nu_x}$ and $\PP_{\lambda,\hat x}$ are statistically ``close'' in an appropriate sense. We approach this using a standard tool: the second moment of the likelihood ratio (closely related to the \emph{chi-squared divergence}). Specifically, we will show the bound \begin{equation}\label{eq:lr-bound}
\left\|
\frac{d\PP_{\lambda,\hat x\to \nu_x}}{d\PP_{\lambda,\hat x}}
\right\|_{\PP_{\lambda,\hat x}}^2
\le
\exp(CB^2)
\end{equation}
for a constant $C = C(\lambda,\pi)$. Let us first see how this bound implies the desired result. Because the event \(A\) is permutation-invariant, \(\PP_{\lambda,z}(A)\) depends only on the empirical distribution of \(z\). Since every realization of \(y\) has empirical distribution \(\nu_x\), this gives
\[
\PP_{\lambda,\hat x\to \nu_x}(A)=\EE_y[\PP_{\lambda,y}(A)]=\PP_{\lambda,x}(A).
\]
Finally, Cauchy--Schwarz gives
\[
\PP_{\lambda,x}(A)
=
\PP_{\lambda,\hat x\to \nu_x}(A)
= \left\langle \frac{d\PP_{\lambda,\hat x\to \nu_x}}{d\PP_{\lambda,\hat x}}, \mathbf{1}_A \right\rangle_{\PP_{\lambda,\hat x}}
\le
\left\|
\frac{d\PP_{\lambda,\hat x\to \nu_x}}{d\PP_{\lambda,\hat x}}
\right\|_{\PP_{\lambda,\hat x}}
\PP_{\lambda,\hat x}(A)^{1/2}
\le
\exp(CB^2)\,\PP_{\lambda,\hat x}(A)^{1/2},
\]
as claimed. Before proving~\eqref{eq:lr-bound}, we remark on some specifics of this strategy.

\begin{remark}\label{rem:lr-choice}
The specific choice of second moment in~\eqref{eq:lr-bound} is somewhat delicate. First, it is important that the distribution in the denominator has a deterministic signal rather than a mixture, in order to calculate the second moment in a tractable way. For the numerator, one could imagine other choices for the distribution, e.g., a uniform mixture over all signals with empirical distribution $\nu_x$, or a fixed signal that has empirical distribution $\nu_x$ and is ``close'' to $\hat{x}$. Neither of these choices gives a small enough value for the second moment. It is important that the distribution in the numerator is a mixture over many signals, all of which are ``close'' to $\hat{x}$.
\end{remark}

It remains to prove~\eqref{eq:lr-bound} by bounding the second moment. Recall the notation $\mu_\lambda(x)$ from \Cref{sec:lse-overview}. A standard calculation with Gaussian densities (see e.g.\ Lemma~1 of~\cite{BMVVX}) shows that if \(y,y'\) are two independent draws of the random vector $y$ described above,
\[
\left\|
\frac{d\PP_{\lambda,\hat x\to \nu_x}}{d\PP_{\lambda,\hat x}}
\right\|_{\PP_{\lambda,\hat x}}^2
=
\EE_{y,y'} \exp\!\left(\left\langle \mu_\lambda(y)-\mu_\lambda(\hat x),\mu_\lambda(y')-\mu_\lambda(\hat x)\right\rangle\right).
\]
Let \(R'\) be the analogue of \(R\) in the generation of \(y'\). Because \(\nu_x\) and \(\nu_{\hat x}\) are \(B\)-typical,
\[
|R|=|R'|=\sum_{i=1}^m \Delta_i
\le
\sum_{i=1}^m n|\nu_{\hat x}(s_i)-\nu_x(s_i)|
\le 2mB\sqrt n.
\]
Since \(y\) and \(\hat x\) differ only on the coordinates in \(R\), every nonzero coordinate $(i,j)$ (or $i \equiv (i,i)$) of \(\Phi(y)-\Phi(\hat x)\) has at least one endpoint ($i$ or $j$) in \(R\). Likewise, every nonzero coordinate of \(\Phi(y')-\Phi(\hat x)\) has at least one endpoint in \(R'\). Therefore, every coordinate that is nonzero in both differences lies in one of the four regions
\[
R\times R',
\qquad
R'\times R,
\qquad
(R\cap R')\times [n],
\qquad
[n]\times (R\cap R').
\]
The total number of such coordinates is at most \(2(|R|^2+n|R\cap R'|)\), and every entry of \(\Phi(y)-\Phi(\hat x)\) and \(\Phi(y')-\Phi(\hat x)\) is \(O(1)\). Here and in the following, $O(\cdot)$ hides a constant factor depending only on $\lambda$ and $\pi$ (with no requirement that $n$ be large). Hence
\[
\left|
\left\langle \Phi(y)-\Phi(\hat x),\Phi(y')-\Phi(\hat x)\right\rangle
\right|
\le
O\bigl(|R|^2+n|R\cap R'|\bigr).
\]
Defining
\[
\widetilde \mu_\lambda(z):=\frac{\lambda}{\sqrt n}\Phi(z),
\]
this yields
\begin{equation}\label{eq:before-normalize}
\left|
\left\langle
\widetilde \mu_\lambda(y)-\widetilde \mu_\lambda(\hat x),
\widetilde \mu_\lambda(y')-\widetilde \mu_\lambda(\hat x)
\right\rangle
\right|
\le
O\bigl(B^2+|R\cap R'|\bigr).
\end{equation}

Next we bound the normalization error between \(\widetilde \mu_\lambda(z)\) and \(\mu_\lambda(z)\). Write
\[
\mu_\lambda(z)=\widetilde \mu_\lambda(z)+r_\lambda(z),
\qquad
r_\lambda(z)
=
\frac{\lambda}{\sqrt n}\left(\frac{n}{\|z\|_2^2}-1\right)\Phi(z).
\]
By \Cref{prop:finite-support-norm} and our choice of \(b\), every vector whose empirical distribution is either \(\nu_x\) or \(\nu_{\hat x}\) satisfies \(|\|z\|_2^2-n|\le O(B\sqrt n)\) and \(\|z\|_2^2\ge n/2\). Therefore
\[
\left|\frac{n}{\|z\|_2^2}-1\right|\le O\!\left(\frac{B}{\sqrt n}\right).
\]
Also, every entry of \(\Phi(z)\) is \(O(1)\), so
\[
\left|\left\langle \Phi(y)-\Phi(\hat x),\Phi(z)\right\rangle\right|
\le O(n|R|) \le O(Bn^{3/2}) \qquad \text{for } z\in\{y',\hat x\}.
\]
It follows that
\[
\left|
\left\langle \widetilde \mu_\lambda(y)-\widetilde \mu_\lambda(\hat x),r_\lambda(z)\right\rangle
\right|
\le
O(B^2) \qquad \text{for } z\in\{y',\hat x\},
\]
and similarly
\[
\left|
\left\langle r_\lambda(z),r_\lambda(z')\right\rangle
\right|
\le
O(B^2) \qquad \text{for } z,z'\in\{y,y',\hat x\}.
\]
Set
\[
D:=\widetilde \mu_\lambda(y)-\widetilde \mu_\lambda(\hat x),
\qquad
D':=\widetilde \mu_\lambda(y')-\widetilde \mu_\lambda(\hat x),
\]
and
\[
E:=r_\lambda(y)-r_\lambda(\hat x),
\qquad
E':=r_\lambda(y')-r_\lambda(\hat x).
\]
From~\eqref{eq:before-normalize} we have
\[
|\langle D,D'\rangle|\le O(B^2+|R\cap R'|),
\]
and the subsequent calculations show
\[
|\langle D,E'\rangle|,\quad
|\langle E,D'\rangle|,\quad
|\langle E,E'\rangle|
\quad\le O(B^2),
\]
where the bound for \(\langle E,D'\rangle\) uses the same argument with \(y\) and \(y'\) interchanged. Expanding
\[
\left\langle \mu_\lambda(y)-\mu_\lambda(\hat x),\mu_\lambda(y')-\mu_\lambda(\hat x)\right\rangle
=
\langle D+E,D'+E'\rangle
\]
therefore gives
\[
\left|
\left\langle \mu_\lambda(y)-\mu_\lambda(\hat x),\mu_\lambda(y')-\mu_\lambda(\hat x)\right\rangle
\right|
\le O\bigl(B^2+|R\cap R'|\bigr).
\]
Hence
\[
\left\|
\frac{d\PP_{\lambda,\hat x\to \nu_x}}{d\PP_{\lambda,\hat x}}
\right\|_{\PP_{\lambda,\hat x}}^2
\le
\exp(C'B^2)\,\EE\exp(C'|R\cap R'|)
\]
for a constant $C' = C'(\lambda,\pi) > 0$.

It remains to bound the exponential moment of $|R \cap R'|$. Since \(|R\cap R'|=\sum_{i\in J}|R_i\cap R_i'|\) and the summands are independent, it suffices to work one \(i\) at a time. Write \(H_i:=|R_i\cap R_i'|\) and \(n_i:=|I_i|=n\nu_{\hat x}(s_i)\). For every \(u\ge 0\),
\[
\EE[(1+u)^{H_i}]
=
\sum_{k=0}^{\Delta_i} u^k \, \EE\binom{H_i}{k}.
\]
For fixed \(k\), the random variable \(\binom{H_i}{k}\) counts the common \(k\)-subsets of \(R_i\) and \(R_i'\). Since \(R_i\) and \(R_i'\) are independent uniformly random \(\Delta_i\)-subsets of \(I_i\),
\[
\EE\binom{H_i}{k}
=
\frac{\binom{\Delta_i}{k}^2}{\binom{n_i}{k}}
\le
\frac{1}{k!}\left(\frac{\Delta_i^2}{n_i}\right)^k.
\]
Therefore
\[
\EE[(1+u)^{H_i}]
\le
\sum_{k=0}^{\infty}
\frac{1}{k!}\left(u\frac{\Delta_i^2}{n_i}\right)^k
=
\exp\!\left(u\frac{\Delta_i^2}{n_i}\right).
\]
Choosing $u$ to be a large enough constant depending on $C'$,
\[
\EE\exp(C' H_i) \le \exp\!\left(O\!\left(\frac{\Delta_i^2}{n_i}\right)\right).
\]
Note that \(\Delta_i\le 2B\sqrt n\). Since \(B\le b\sqrt n\) and \(\nu_{\hat x}\) is \(B\)-typical, our choice of \(b\) gives \(n_i\ge n\pi(s_i)/2\). Because \(m\) is fixed,
\[
\sum_{i\in J}\frac{\Delta_i^2}{n_i}\le O(B^2).
\]
Thus
\[
\EE\exp(C'|R\cap R'|)
=
\prod_{i\in J}\EE\exp(C' H_i)
\le
\exp(O(B^2)).
\]
We conclude
\[
\left\|
\frac{d\PP_{\lambda,\hat x\to \nu_x}}{d\PP_{\lambda,\hat x}}
\right\|_{\PP_{\lambda,\hat x}}^2
\le
\exp(O(B^2)),
\]
which proves~\eqref{eq:lr-bound}.
\end{proof}

\phantomsection
\addcontentsline{toc}{section}{Acknowledgments and Disclosure on AI Usage}

\section*{Acknowledgments and Disclosure on AI Usage}

ASW thanks Ji Oon Lee for answering some questions about linear spectral statistics.

The spiked Wigner result was discovered without the help of AI, while the computational Neyman--Pearson lemma --- particularly \Cref{lem:direct-roc-to-ratio} --- was proved in collaboration with GPT 5.5. GPT 5.5 has also suggested some slick proofs for routine technical lemmas, most notably the bound on moments of a normalized rank-1 signal (\Cref{lem:normalized-monomial-bound}). Some AI tools were used in the writing process, but the authors have verified --- and largely rewritten --- the AI output. The authors take full responsibility for the contents of this paper.

\phantomsection
\addcontentsline{toc}{section}{References}

\bibliographystyle{alpha}
\bibliography{main}

\end{document}